\documentclass[12pt]{article}
\usepackage{amsmath,amsfonts,amssymb,amsthm}
\usepackage{enumitem} 
\usepackage{graphicx}
\usepackage{color}
\newcommand{\fed}{\,\rule{.1mm}{.24cm}\rule{.24cm}{.1mm}\,}

\newcommand{\nc}{\newcommand}
\newtheorem{thm}{Theorem}
\newtheorem{lem}[thm]{Lemma}

\newtheorem{cor}[thm]{Corollary} 
\nc{\R}{{\mathbb R}^d} 
\nc{\N}{{\mathbb N}} 
\nc{\K}{{\cal K}}

\DeclareMathOperator{\reg}{\rm reg} 
\DeclareMathOperator{\Nor}{\rm Nor}
 
\DeclareMathOperator{\inter}{\rm int}

\nc{\kpo}{(\K^d_{2,gp})_o}
\nc{\krep}{\K^d_{3,gp}} 
\nc{\BP}{\mathbb{P}} 
\nc{\BQ}{\mathbb{Q}}
\nc{\BE}{\mathbb{E}}
\nc{\dF}{\partial F} 
\nc{\ve}{\varepsilon}
\nc{\V}{{\cal V}_{\ve}(\dF)}

\begin{document}

\title {Differentiation of sets - The general case} 
\author{Est\'ate V. Khmaladze \and Wolfgang Weil}
%\centerline{Victoria University of Wellington}

\date{\today}
%(School of Mathematics, Statistics and Computer Science, P.O.Box 600, Wellington, New Zealand, e-mail: Estate.Khmaladze@vuw.ac.nz)

%\noindent Postal address:\\School of Mathematics, Statistics and Computer Science\\
%Victoria University of Wellington\\P.O.Box 600, Wellington\\New Zealand\\
%E-mail: Estate.Khmaladze@vuw.ac.nz\\Ph:  +64 (4) 463 56 52\\Fax:+64 (4) 463 50 45\\

\maketitle
\begin{abstract} In recent work by Khmaladze and Weil \cite{KW} and by Einmahl and Khmaladze \cite{EiK}, limit theorems were established for local empirical processes near the boundary of compact convex sets $K$ in $\R$. The limit processes were shown to live on the normal cylinder $\Sigma$ of $K$, respectively on a class of set-valued derivatives in $\Sigma$. The latter result was based on the concept of differentiation of sets at the boundary $\partial K$ of $K$, which was developed in Khmaladze \cite{Khm07}. Here, we extend the theory of set-valued derivatives to boundaries $\partial F$ of rather general closed sets $F\subset \R$, making use of a local Steiner formula for closed sets, established in Hug, Last and Weil \cite{HLW}.
\end{abstract}
%{\bf Mathematical Subject Classification (1991).} 28C99, 28E05, 46G05, 53C65, 60D05, 60G55.
%{\bf Key words:} local Steiner formula; local point process; set-valued mapping.}

\section{Introduction}

\noindent
The general aim of this work is to describe infinitesimal changes in the shape of a set in $\R$ through an appropriate notion of a {\it derivative set}. Namely, if bounded sets $F(\ve)\subset\R$ converge, as $\ve\to 0$, to a given set $F$, then we want to say what is the derivative of $F(\ve)$, at $\ve =0$. We hereby extend the approach, which was developed in \cite{Khm07} under convexity assumptions.

This line of research  is motivated by a class of problems in spatial statistics. To be more precise, consider a set $A\subset \R$ marking the boundary between two regions in $\R$ which carry two different probability distributions. Given $n$ random points $\xi_1,...,\xi_n$ chosen independently from the compound distribution in $\R$, the statistical challenge is to draw information about the geometry of $A$ from the empirical process given by the $\xi_i$.  This {\it change set} problem is a natural generalization of the {\it change point} problem on the real line (where $A$ consists of one point only), a classical problem in statistics (see, e.g., \cite{CMS, BD}). The change set problem is of a more recent nature (cf. \cite{KhMT1, KhMT2, IM, KTs}). For the case where $A=\partial K$ is the boundary of a convex body $K$ (a compact convex set in $\R$), the local empirical process in the neighborhood of $\partial K$ was studied in Khmaladze and Weil \cite{KW} and a Poisson limit result was established, as the neighborhood shrinks. The approach made use of a Steiner formula for support measures (curvature measures), which sit on the normal bundle of $K$, and the limit process was shown to live on the corresponding normal cylinder. More recently, Einmahl and Khmaladze \cite{EiK} proved a central limit theorem for such local empirical processes. The Gaussian limit process which they established sits on certain derivative sets in the normal cylinder. This approach required the notion of derivative of sets in measure, a concept which was developed in Khmaladze \cite{Khm07}. 

Indeed, if a particular choice of a region $K$ is considered as a hypothesis, then the challenging problem is to distinguish, by statistical methods, between this $K$ and a class of possible small deformations $\tilde K$  of $K$. It is natural to describe {\color{red} each such deformation} $\tilde K= K(\ve)$ as a set-valued function, converging to $K$ as $\ve\to 0$. As a stable trace of the deviation $K(\ve)\triangle K$ of $K(\ve)$ from $K$, it is consequent to establish a derivative of $K(\ve)$ at $K$ as a set in a properly chosen domain. The local point processes in the neighborhood of the boundary $A=\partial K$ will live asymptotically on the class of such derivative sets, as was shown in \cite{EiK}, \cite{KW}. Derivative sets of this type are of interest in infinitesimal image analysis in general.

It should be mentioned that the differentiation of set-valued functions is a well-established field of research and prominent concepts, much older than that of \cite{Khm07}, exist. In particular, the tangent cone approach is described in Aubin and Frankowska \cite{AF} and Borwein and Zhu \cite{BZ} and provides a classical tool in this field. A much advanced form of affine mappings, the multi-affine mappings of Artstein \cite{Art} along with the quasiaffine mappings of Lemar\'echal and Zowe \cite{LZ} demonstrate another approach to the differentiability of sets.

So far, in the papers \cite{Khm07}, \cite{EiK}, \cite{KW} mentioned above, the basic set $K$ was assumed to be compact and convex. This provided a convenient geometric situation. The set had a well defined outer and inner part, each boundary point had at least one outer normal, the boundary and the normal bundle had finite $(d-1)$-dimensional Hausdorff measure ${\cal H}^{d-1}$, the normal cylinder had an unbounded upper part and a bounded lower part, and the support measures were finite and nonnegative. For applications, of course, more general set classes would be interesting. Some generalizations, for example to polyconvex sets (finite unions of convex bodies) or to sets of positive reach, are possible with minor modifications. In the following, we aim for a rather general framework allowing closed sets with only few topological regularity properties and we discuss the differentiation of such sets in the spirit of \cite{Khm07}. In the background is a general Steiner formula for closed sets, established in \cite{HLW}, which we will use intensively. 

General closed sets $F\subset \R$ can have quite a complicated structure. They need not have a defined inner and outer part. Even in the compact case, their boundary can have  infinite Hausdorff measure ${\cal H}^{d-1}(\partial F)$ or even positive Lebesgue measure $\mu_d (\partial F)>0$. Boundary points $x\in\partial F$ need not have any normal, but also can have one, two or infinitely many normals. Consequently, the normal bundle $\Nor (F)$ of $F$ (or $\Nor (\partial F)$ of $\partial F$), as it was defined in \cite{HLW}, can also have a rather complicated structure. Moreover, the support measures of $F$, which were introduced in \cite{HLW} as ingredients of the general Steiner formula, are signed Radon-type measures. They are finite only on sets in the normal bundle with local reach bounded from below (see Section 2, for detailed explanations). In our attempt to define the derivative of a family $F(\varepsilon)$ at a set $F$, we therefore concentrate on two important situations, which simplify the presentation but are still quite general. First, in Section 3, we consider compact sets $F$ which are the closure of their interior and satisfy $\mu_d(\partial F) =0$. We call these {\it solid} sets. Second, in Section 4, we discuss {\it boundary} sets $F$. These are compact sets without interior points and with $\mu_d (F)=\mu_d(\partial F)=0$. Based on these two set classes, we then study, in Section 5, a differentiation were bifurcation in a set-valued function may occur. The next section, Section 6, investigates some important examples of set functions which are differentiable in our sense, namely families $F(\varepsilon)$ which arise as local or global (outer) parallel sets. In the final section, we discuss some variants of the differentiability concept. We start in Section 2 with collecting the necessary notations and preliminary results.

\section{Preliminaries }

\noindent
In the following,  $F$ is a nonempty closed set in ${\R}$ and $\dF$ denotes its boundary.
For $z\in\R $, let $p(z)=p_{F}(z)$ be the metric projection of $z$ onto $F$, that is, 
the point in $F$ nearest to $z$,
$$
\| z - p(z)\| = \min_{x\in F} \|z - x\|,
$$
and let $d(z)=d_F(z)=\|z-p(z)\|$ be the distance from $z$ to $F$. For $\ve >0$, the {\it $\ve$-neighborhood} $F_\ve$ of $F$ is defined as
$$
F_\ve = \{ z\in\R : d(z)\le\ve\}. 
$$
The {\it skeleton} of $F$ is the set 
$$
S_{F} =\{z \in {\R}: p(z) {\rm \; is \; not \; unique}\}.
$$
It is known that  $\mu_d(S_{F}) = 0$, where  $\mu_d$ is the Lebesgue measure in $\R$ (see \cite{HLW}). If $z\notin S_F\cup F$, then $p(z)\in\dF$ and we let $u(z)=u_F(z)$ be the corresponding direction, namely the vector in the unit sphere $S^{d-1}$ given by
$$
u(z) = \frac{z-p(z)}{d(z)} .
$$
We call $u=u(z)$ an {\it (outer) normal} of $F$ in $x=p(z)$. Note that a point $x\in \dF$ can have more than one normal (we denote by $N(x)$ the set of all normals in $x$) and that also some points $x\in\dF$ may not have any normal. In that case, we put $N(x)=\emptyset$.

The {\it (generalized) normal bundle} $\Nor(F)$ of $F$ is the subset of $\dF\times S^{d-1}$ defined as
$$
\Nor(F) =\{ (x,u): x \in \dF, \; u \in N(x) \}.
$$ 
Thus, $\Nor (F)$ consists of all pairs $(x,u)$ for which there is a point $z\notin S_F\cup F$ with $x=p(z)$ and $u=u(z)$. Such a point is then of the form $z=x+tu$ with $t=d(z)>0$. Since the ball $B(x+tu,t)$ touches $F$ only in the point $x$, this implies that the whole segment $[x,x+tu]$ projects (uniquely) onto $x$. This fact gives rise to the {\it reach function} $r=r_F$ of $F$, which is defined on $\Nor (F)$,
$$
r(x,u) =\sup \{s>0 : p(x+su) = x \}.
$$
Note that in \cite{HLW}, a reach function $\delta$ on $\Nor (F)$ was defined in a slightly different way  (by $\delta(x,u) =\inf \{s>0 : x+su \in S_F \}$).
It is easy to see that $r\le\delta$ and J. Kampf (unpublished) gave an example of a set $F$ and a pair $(x,u)\in\Nor (F)$ such that $r(x,u)<\delta (x,u)$. In the following main result from \cite{HLW}, the {\it local Steiner formula}, $\delta$ appeared in the statement in \cite{HLW}, but the correct reach function $r$ was used in the proof. 

Before we can formulate the result, we need to recall from \cite{HLW} the notion of a {\it reach measure} $\Theta (F,\cdot)$ of $F$. For $(x,u)\in \Nor (F)$, let $h(x,u)\in [0,\infty ]$ be defined by
$$
h(x,u) =  \max \{\| x\|, r (x,u)^{-1}\} .
$$
A subset $A\subset \Nor (F)$ is {\it $h$-bounded} if $A\subset \{ h\le c\}$, for some $0\le c<\infty$. A {\it signed $h$-measure} $\Theta$ is then a set function with values in $[-\infty, \infty]$, defined on the system of $h$-bounded Borel sets  in $\Nor (F)$ and such that the restriction of $\Theta$ to each set $\{ h\le c\}$,  $0\le c<\infty$, is a signed measure of finite variation. For a signed $h$-measure, the Hahn-decomposition on each set $\{ h\le c\}$ leads to  a unique representation $\Theta = \Theta^+-\Theta^-$ with mutual singular $\sigma$-finite measures $\Theta^+, \Theta^-\ge 0$ which are finite on each sublevel set $\{ h\le c\}$,  $0\le c<\infty$. $\Theta^+,\Theta^-$ and the total variation measure $|\Theta| = \Theta^++\Theta^-$ can then be extended (in a unique way) to all Borel sets in $\Nor (F)$, but this is not possible, in general, for $\Theta$. Instead of a signed $h$-measure $\Theta$ we speak of an {\it $r$-measure} (reach measure) $\Theta (F,\cdot)$ in the following and we call Borel sets $A\subset \Nor(F)$ {\it $r$-bounded} if they are $h$-bounded, for the specific function $h$ defined above. We also write $|\Theta|(F,\cdot)$ for the variation measure.

We denote the minimum of $a,b\in{\mathbb R}$ by $a\wedge b$.

\begin{thm}[\cite{HLW}]\label{SF}
For any non-empty closed set $F\subset\R$, there exist uniquely determined
$r$-measures $\Theta_0(F,\cdot),...,\Theta_{d-1}(F,\cdot)$ of $F$ satisfying
\begin{equation}\label{intcond}
\int_{\Nor (F)} {\bf 1}_B(x)(r( x, u) \wedge c)^{d-i} |\,\Theta_i |(F, d(x, u)) < \infty, 
\end{equation}
for $i = 0, . . . , d - 1$, all compact sets $B\subset \R$ and all $c > 0$, such that, for any
measurable bounded function $f : \R \to {\mathbb R}$ with compact support, we have
\begin{align} \label{steiner1}
\int_{\R\setminus F}& f(z) \mu_d(dz )\cr 
&= \sum_{j=1}^{d} \binom{d-1}{j-1}\int_{\Nor(F)}\int_0^{r(x,u)} f(x+tu) t^{j-1}\, dt\,  \Theta_{d-j}(F,d(x,u)).
\end{align}
\end{thm}

The measures $\Theta_0(F,\cdot),...,\Theta_{d-1}(F,\cdot)$ will be called the {\it support measures} of $F$. This notation is justified by the case of convex bodies (compact convex sets) $F$, where the result is well-known and involves the classical support measures of $F$ (see \cite{Sch}). For convex bodies $F$ the reach function $r$ is infinite, $r(x,u)=\infty$.
The local Steiner formula includes the classical Steiner formula (for convex bodies $F$),
\begin{equation} \label{steiner0}
 \mu_d((F+rB^d)\setminus F ) = \frac{1}{d}\sum_{j=1}^{d} \binom{d}{j} t^{j}\Theta_{d-j}(F,\Nor (F)), 
\end{equation}
where the total measures $\Theta_i(F,\Nor (F)), i=0,...,d-1$, are proportional to the {\it intrinsic volumes} of $F$. 

Whereas, for a convex body $F$, the  $\Theta_i(F,\cdot), i=0,..., d-1,$ are finite (nonnegative) Borel measures on $\Nor (F)$, the situation is  more complicated for closed sets $F$. As we have explained above, the $r$-measures  $\Theta_i(F,\cdot), i=0,..., d-1,$ can attain negative values and are only defined on $r$-bounded sets, in general.  Hence the notion of $r$-measures is similar to the one of signed Radon measures, as they appear in functional analysis. Since the total variation measure $|\Theta_i|(F,\cdot) = \Theta_i^+(F,\cdot)+\Theta_i^-(F,\cdot)$ exists on all Borel sets in $\Nor (F)$, the integrability relation \eqref{intcond} guarantees that the integrals on the right side of \eqref{steiner1} exist (without any restriction) and are finite.
For more details, see \cite{HLW}.

We call a boundary point $x\in\dF$ {\it regular}, if $N(x)$ consists either of one vector $u$ or of two antipodal vectors $u,-u$. Let $\reg(F)$ be the set of regular points of $\dF$.

In the following, we are first interested in closed sets $F$, which are {\it solid} in the sense that $F$ is the closure of its interior and that $\mu_d(\partial F)=0$ holds. For such sets, we will also develop an expansion into the interior. This can be done simply by replacing $F$ by $F^*$, the closure of the complement of $F$. We have
$$
\Nor (\dF) = \Nor (F)\cup \Nor (F^*),\quad \Nor (F)\cap \Nor (F^*) = \emptyset .
$$
This gives rise to the extended normal bundle $\Nor_e (F)$ of $F$ which is the union $\Nor (F) \cup R(\Nor (F^*))$, were $R$ is the reflection $(x,u)\mapsto (x,-u)$. We extend the reach function $r$ of $F$ to the {\it outer reach function} $r_+$ on $\Nor_e (F)$ by putting $r_+(x,u)= r(x,u)$, for $(x,u)\in \Nor (F)$, and $r_+(x,u)= 0$, for $(x,u)\in R(\Nor (F^*))\setminus \Nor (F)$. Correspondingly, we define an {\it inner reach function} $r_-$ of $F$ by $r_- (x,u) = r(F^*,x,-u)$, for $(x,u) \in R(\Nor (F^*))$ and $r_- (x,u) =  0$, for $(x,u) \in \Nor (F)\setminus R(\Nor (F^*))$. The support measures $\Theta_i(F,\cdot), i=0,..., d-1,$ of $F$ can be extended to $\Nor_e(F)$ by putting
$$
\Theta_i(F,\cdot) = (-1)^{d-1-i}\Theta_i(F^*,\cdot)\circ R^{-1}
$$
on $R( \Nor (F^*))$. This definition is consistent since, on the intersection $$\Nor (F)\cap R( \Nor (F^*)),$$ we have $$\Theta_i(F,\cdot)=(-1)^{d-1-i}\Theta_i(F^*,\cdot)\circ R^{-1}$$  (see \cite[Prop. 5.1]{HLW}).

Now the following variant of the local Steiner formula \eqref{steiner1} holds,
\begin{align} \label{steiner2}
\int_{\R\setminus \dF}& f(z) \mu_d(dz )\\
&= \sum_{j=1}^{d} \binom{d-1}{j-1}\int_{\Nor_e(F)}\int_{-r_-(x,u)}^{r_+(x,u)} f(x+tu) t^{j-1} dt \Theta_{d-j}(F,d(x,u)) \nonumber
\end{align}
(see \cite[Th. 5.2]{HLW}). Since we have assumed $\mu_d(\dF)=0$, the integration on the left can be performed over the whole $\R$. Note that $\mu_d(\dF)=0$ must not even hold, if $F$ is the closure of its interior. An example is given by a Cantor-type set in $[0,1]$. As in the classical Cantor set, open intervals are deleted in each step, but such that the total length of all deleted intervals is a constant $c<1$. Let $A$ be the union of all open intervals which are deleted in even-numbered steps and $B$ the corresponding union of the intervals deleted in odd-numbered steps. $A$ and $B$ are disjoint open sets and their (common) boundary is $C=[0,1]\setminus (A\cup B)$ with $\mu_1 (C)=1-c>0$. Moreover, the  sets $A\cup C$ and $B\cup C$ are both the closure of their interior. 

The first order term in \eqref{steiner2} (with respect to $t$) involves the support measure $\Theta_{d-1}(F,\cdot)$. As it follows from \cite[Prop. 4.1]{HLW}, $\Theta_{d-1}(F,\cdot)$ is a nonnegative $\sigma$-finite measure on $\Nor_e(K)$ which, for a solid set $F$, is concentrated on the pairs $(x,u)$ with $x\in \reg_e(F) = \reg (F)\cup \reg (F^*)$ and is given by the Hausdorff measure,
\begin{equation}\label{d-1meas1}
\Theta_{d-1}(F,\cdot) = \int_{\reg_e(F)} {\bf1}\{ (x,\nu (F,x))\in\cdot\}{\cal H}^{d-1}(dx) .
\end{equation}
Here, $\nu(F,x)$ is the normal vector $u\in N(x)$, for which $(x,u)\in\Nor_e(F)$ (for $x\in\reg_e(F)$, this vector $u$ is uniquely determined).
Note that ${\cal H}^{d-1}(\dF\setminus\reg_e(F))>0$ is possible, even for solid sets $F$.
 
For (full dimensional) convex bodies $F$, formula \eqref{steiner2} reduces to Theorem 1 in \cite{KW} (here, the outer reach function $r_+$ is infinite). $F$ is then solid, all support measures are finite and nonnegative and  ${\cal H}^{d-1}$-almost all boundary points $x\in\partial F$ are regular.
%==========================new section 3============================

\section{Definition of differentiability: The case of solid sets}

Throughout this section, we assume that $F\subset \R$ is compact and solid (hence nonempty with $\mu_d (\dF)=0$). Since the following notions and results are of a local nature, they can be generalized appropriately to unbounded solid sets $F$ using intersections with a family of growing balls.

The differentiation procedure, as it was introduced in \cite{Khm07}, lives on the {\it normal cylinder} $\Sigma = \Sigma(F)$ which, in the case of solid $F$, is defined as $\Sigma = {\mathbb R} \times \Nor_e(F)$.

For $\ve > 0$, we define the {\it local magnification map} $\tau_{\ve}$ as a mapping from $\R\setminus (S_{\dF}\cup \dF)$ to $\Sigma$ by 
$$      
\tau_{\ve}(z)  = (\frac{d(z)}{\ve},p(z),u(z)),
$$
for $z\in \R\setminus (S_F\cup F)$, and
$$      
\tau_{\ve}(z)  = (-\frac{d(z)}{\ve},p(z),-u(z)),
$$
for $z\in \R\setminus (S_{F^*}\cup F^*)$.

\begin{lem} \label{thm:pre} $ \tau_{\ve}$ is a bicontinuous one-to-one mapping from ${\R}\setminus (S_{\dF}\cup \dF )$ to 
$$
\left\{(t,x,u) : (x,u) \in \Nor_e(F), t\in (-\frac{r_-(x,u)}{\ve},0)\cup (0,\frac{r_+(x,u)}{\ve})\right\}\subset \Sigma.  
$$
\end{lem}

In the following, we apply $\tau_\ve$ to arbitrary Borel sets $A\subset \R$,
$$
\tau_\ve(A)=\{\tau_\ve (x) : x\in A\setminus (S_{\partial F}\cup\partial F)\} .
$$
By Lemma \ref{thm:pre}, $\tau_\ve(A)$ is then a Borel set.

Now, consider a set-valued mapping $F(\ve), 0\leq  \ve \leq 1$, such that  $F(0)=F$ (we imagine all the sets $F(\ve)$ to be nonempty compact, but actually, for $\ve >0$, bounded Borel sets $F(\ve)$ would also work). It is natural to expect that a notion of differentiability of $F(\ve)$ at $F$ should be equivalent to the differentiability of
$F(\ve) \triangle F$ at  $\dF$. 

Therefore, we start with an arbitrary family $A(\ve), 0\leq \ve \leq 1,$ of  Borel sets such that $A(0)\subset \dF$. We call the family $A(\ve), 0\leq \ve \leq 1,$ {\it essentially bounded} (with bound $T$), if there is some $T>0$ such that
\begin{equation}\label{bound}
\frac{1}{\ve} \mu_d (A(\ve)\cap (\R\setminus (\dF)_{\ve T})) \to 0\quad{\rm as\ }\ve\to 0.
\end{equation}

We also need the measure $M= M_F=\mu_1\otimes \Theta_{d-1}(F,\cdot )$ on $\Sigma$.
 
%=============================Def 1 &  2===========================
\bigskip\noindent
{\bf Definition 1.} The set-valued mapping $A(\ve), 0\leq \ve \leq 1,$ is {\it differentiable} at $\dF$ for $\ve=0$, if it is essentially bounded and if there exists a Borel set $B \subset \Sigma$ such that $M(\tau_{\ve}(A(\ve))\triangle B) \to 0$, as $\ve\to 0$.  The set $B$ is then called the {\it derivative} of $A(\ve)$ at $\dF$ (for $\ve =0$).\\

\noindent
{\bf Definition 2.} The set-valued mapping $F(\ve), 0\leq \ve \leq 1,$ is {\it differentiable} at $F$ for $\ve=0$, if $A(\ve)=F(\ve)\triangle F$ is differentiable at $\dF$. The {\it derivative} of $F(\ve)$ at $F$ is then defined to be the same as the derivative of
$A(\ve)$ at $\dF$.\\
In notations
$$
 \frac{d}{d\ve}F(\ve)|_{\ve=0} = \frac{d}{d\ve}A(\ve)|_{\ve=0} = B.
$$

\medskip
Note that the set $B$ is not unique, but can be changed on a set of $M$-measure $0$. If $A(\ve)$ is differentiable at $\dF$, then  $\tilde A(\ve) = A(\ve )\cap (\dF)_{\ve T}$ is differentiable at $\dF$. We therefore
{\color{red} can} assume, without loss of generality, that $A(\ve)\subset (\dF)_{\ve T}$. Moreover, if $T$ is the bound in \eqref{bound}, we can assume 
$$
B\subset \Sigma_T = \{ (t,x,u)\in \Sigma : -T\le t\le T\} .
$$
By construction, the differentiability of $A(\ve)$ only depends on the behavior outside $\partial F$. Hence, we may also assume $A(\ve)\cap \partial F=\emptyset$, $0\le\ve\le 1$, if this is helpful. In particular, we then have $A(0)=\emptyset$. For the differentiability of $F(\ve)$ at $F$, this means that we can replace $F(\ve)\triangle F$ by $(F(\ve)\setminus F)\cup(\inter F\setminus F(\ve))$.

As a simple example, we mention the constant mapping $F(\ve) =F, 0\le\ve\le 1$. Since $A(\ve)= F\triangle F=\emptyset$ is differentiable at $\partial F$ with derivative $B=\emptyset$, $F(\ve)$ is differentiable at $F$ with derivative $\emptyset$.

\medskip
The next lemma shows some algebraic properties of the differentiation. In its formulation, for a set $C\subset\Sigma$, we put
$$
C^+=\{ (t,x,u)\in C : t\ge 0\}
$$
and
$$
C^-=\{ (t,x,u)\in C : t< 0\}.
$$
%===============new Lemma 2===============================
\noindent
\begin{lem}\label{thm:sets} (i) If $A_1(\ve)$ and $A_2(\ve)$ are differentiable at $\dF$ and $B_1$ and $B_2$
are corresponding derivatives, then 
$A_1(\ve)\cup A_2(\ve)$, $A_1(\ve)\setminus A_2(\ve)$ and $A_1(\ve)\cap A_2(\ve)$ are also differentiable at $\dF$ and the derivatives
are $B_1\cup B_2$, $B_1\setminus B_2$ and $B_1\cap B_2$ respectively.

(ii) If $F_1(\ve)$ is differentiable at $F$ and $A_2(\ve)$ is differentiable at $\dF$ and $B_1$ and $B_2$ are corresponding derivatives, then $F_1(\ve)\cup A_2(\ve)$ is differentiable at $F$ and the derivative is $B$ with $B^+=B_1^+\cup B_2^+$ and $B^-
=B_1^-\setminus B_2^-$. At the same time $F_1(\ve)\setminus A_2(\ve)$ is also
differentiable at $F$ and the derivative is $B$ with $B^+=B_1^+\setminus B_2^+$
and $B^-=B_1^-\cup B_2^-$.

(iii) For $a\in {\mathbb R}$ and $B\subset\Sigma$ define  $aB=\{(as,x,u): (s,x,u)\in B\}$. Let $\ve\mapsto f(\ve)$ be a non-negative function differentiable at $0$ and $f(0)=0$. If $F(\ve)$ is differentiable at $F$ with derivative $B$, then $F(f(\ve))$ is also differentiable at $F$ and the derivative is $f'(0)B$. \\
\end{lem}

\begin{proof} See \cite[Lemma 2]{Khm07}.
\end{proof}

\medskip
Suppose $\BP$ is an absolutely continuous measure on $\R$ with density $f\ge 0$. We would like to require that  $f(z)$ can be approximated in the neighborhood
of $\dF$ by a function depending on $p_{\dF}(z)$ only.  However, it is possible that
the approximating functions are different for $z$ tending to $p_{\dF}(z)$ from outside
$F$ and from inside $F$. Hence our formal requirement is that there are two bounded measurable functions $\bar f_+\ge 0$  and $\bar f_-\ge 0$ on $\dF$, such that
\begin{align}\label{lambda}
&\frac{1}{\ve}\int_{\R} {\bf 1}\{ 0<d(F,z)\le\ve\}|f(z) - \bar f_+ (p_{F}(z))| \mu_d(dz)  
\to 0 , \cr
&\frac{1}{\ve}\int_{\R} {\bf 1}\{ 0<d(F^*,z)\le\ve\}|f(z) - \bar f_- (p_{F^*}(z))| \mu_d(dz)
\to 0 ,
\end{align}
as $\ve\to 0$. 
Now define a measure $\BQ$ on $\Sigma$ as follows:
\begin{align*}%\label{Q}
\BQ(d(s,x,u)) &= ds \times \bar f_+(x)\Theta_{d-1}(F,d(x,u)) \; \; {\rm on} \; \; \Sigma^+, \cr
\BQ(d(s,x,u)) &= ds \times \bar f_-(x)\Theta_{d-1}(F,d(x,u)) \; \; {\rm on} \; \; \Sigma^-.
\end{align*}
Here,
$$
\Sigma^+ = \{(s,x,u)\in\Sigma : s\ge 0\},\ \Sigma^- = \{(s,x,u)\in\Sigma : s< 0\}.
$$

%==============new theorem 3 on measures===============================
\noindent
\begin{thm}\label{thm:mes}  Suppose that the measure $\BP$ satisfies condition \eqref{lambda} and suppose that the functions $\bar f_-, \bar f_+$ are integrable with respect to $|\Theta_{i}|(F,\cdot ), i=0,\dots,d-1$. Let also $A(\ve)\subset (\dF)_{\ve T}$ (for some $T>0$) be a set-valued mapping which is differentiable at $\dF$ (with derivative $B\subset\Sigma_T$). 
Then
\begin{equation}\label{difmes}
\frac{d}{d\ve}\BP(A({\ve}))|_{\ve=0} = \BQ(\frac{d}{d\ve}A({\ve}) |_{\ve = 0}) = \BQ(B) .
\end{equation}
\end{thm}

\begin{cor}\label{thm:mes2} Suppose that the conditions of the theorem hold for $A(\ve)=F(\ve)\triangle F$. Then
$$
\frac{d}{d\ve}\BP(F({\ve}))|_{\ve=0} = \BQ(\frac{d}{d\ve}A^{+}({\ve}) |_{\ve = 0})
- \BQ(\frac{d}{d\ve}A^{-}({\ve}) |_{\ve = 0}),
$$
where $A^+(\ve ) = F(\ve)\setminus F$ and $A^-(\ve ) = F\setminus F(\ve)$. 
\end{cor}

\medskip\noindent
{\it Proof of Theorem \ref {thm:mes}.} We may assume that $T=1$.

Since $\BP(A({0}))=0$ (due to our assumption $\mu_d(\dF)=0$), we have to establish the  asymptotic behaviour of  $\ve^{-1}\BP(A({\ve}))$. 

We consider an auxiliary measure $\bar \BP$ on $(\dF)_{\ve}$ with density $\bar f_{+}(p_{\dF}(z))$, respectively $\bar f_{-}(p_{\dF}(z))$, according to $z\in F_{\ve}\setminus F$ or $z\in F^*_{\ve}\setminus F^*$. Condition \eqref{lambda} implies that $\ve^{-1}[\BP (A(\ve)) - \bar \BP (A(\ve))] \to 0$,  hence we can concentrate  on $\ve^{-1}\bar \BP (A(\ve))= \ve^{-1}\bar \BP (A^+(\ve))+\ve^{-1}\bar \BP (A^-(\ve))$, where $A^+(\ve) = A(\ve)\setminus F, A^-(\ve) = A(\ve) \cap F$. 

Since 
$$
\bar \BP (A^+(\ve))  = \int_{\R\setminus \dF} \bar f_+(p_{\dF}(z)){\bf 1}_{A^+(\ve)}(z)\, \mu_d(dz)
$$
and $z\mapsto \bar f_+(p_{\dF}(z)){\bf 1}_{A^+(\ve)}(z)$ is bounded with compact support, we can apply the local Steiner formula \eqref{steiner2}. It follows that
\begin{align}\label{asymp}
&\bar \BP (A^+(\ve))  = \int_{\Nor_e(F)}\int_{0}^{r_+(x,u)\wedge\ve} \bar f_{+}(x) {\bf 1}_{A^{+}(\ve)}(x+tu) dt \, \Theta_{d-1}(F,d(x,u))\\
&+ \sum_{j=2}^{d} {\binom{d-1}{j-1}} \int_{\Nor_e(F)}\int_{0}^{r_+(x,u)\wedge\ve} \bar f_{+}(x) {\bf 1}_{A^{+}(\ve)}(x+tu) t^{j-1} dt \, \Theta_{d-j}(F,d(x,u)).\nonumber 
\end{align}
The sum of the higher order terms is $o(\ve)$. Indeed, for each integral we have
\begin{align*}
&\left|\int_{\Nor_e(F)}\int_{0}^{r_+(x,u)\wedge\ve} \bar f_{+}(x) {\bf 1}_{A^{+}(\ve)}(x+tu) t^{j-1} dt \, \Theta_{d-j}(F,d(x,u))\right|\cr
&\quad\leq \int_{\Nor_e(F)}\bar f_{+}(x)\left(\int_{0}^{r_+(x,u)\wedge\ve}  {\bf 1}_{A^{+}(\ve)}(x+tu) t^{j-1} dt\right) \, |\Theta_{d-j}|(F,d(x,u))\cr
&\quad\le \frac{\ve^{j}}{j}\int_{\Nor_e(F)}\bar f_{+}(x)\, |\Theta_{d-j}|(F,d(x,u))
\end{align*}
with $j\ge 2$, 
and the latter integral is finite, by our assumptions.

 As to the asymptotic behaviour of the first summand in \eqref{asymp}, we have
\begin{align*}
\frac{1}{\ve}&\int_{\Nor_e(F)}\int_{0}^{r_+(x,u)\wedge\ve} \bar f_{+}(x) {\bf 1}_{A^{+}(\ve)}(x+tu) dt \, \Theta_{d-1}(F,d(x,u))\cr
 &=  \int_{\Sigma} {\bf 1}\{0\le t\le \frac{r_+(x,u)}{\ve}\wedge 1\}\ \bar f_{+}(x) {\bf 1}_{B^{+}(\ve)}(t,x,u)M(d(t,x,u)) 
\end{align*}
with $B^{+}(\ve) = \tau_\ve (A^{+}(\ve))$.
However,
the differentiability of $A(\ve)$ implies that of $A^+(\ve)$ (with limit $B^+$) by Lemma \ref{thm:sets}. Therefore, the function
$|{\bf 1}_{B^{+}(\ve)}(t,x,u) - {\bf 1}_{B^{+}}(t,x,u)|$ tends to $0$ $M-$a.e. on $\Sigma$ and Lebesgue's theorem of majorised convergence implies that
$$
\int_{\Sigma} {\bf 1}\{0\le t\le \frac{r_+(x,u)}{\ve}\wedge 1\}\ \bar f_{+}(x) |{\bf 1}_{B^{+}(\ve)}(t,x,u) - {\bf 1}_{B^{+}}(t,x,u)| M(d(t,x,u))
$$
tends to $0$, as $\ve\to 0$.
This shows that
$$
\frac{1}{\ve}\bar \BP (A^+(\ve))\to \int_{\Sigma} {\bf 1}\{0\le t\le  1\}\ \bar f_{+}(x) {\bf 1}_{B^{+}}(t,x,u)M(d(t,x,u)) .
$$

With respect to $A^-(\ve)$, we can proceed similarly, since
\begin{align*}
\bar \BP (A^-(\ve))  &= \int_{A(\ve)\cap F} {\bf 1}_{A^-(\ve)}(z)\, \bar \BP(dz)\cr
&= \int_{\R\setminus \dF} \bar f_-(p_{\dF}(z)){\bf 1}_{A^-(\ve)}(z)\, \mu_d(dz),
\end{align*}
again due to the assumption that $\mu_d(\dF)=0$. Hence the Steiner formula \eqref{steiner2} can be used again and gives us, as above,
$$
\frac{1}{\ve}\bar \BP (A^-(\ve))\to  \int_{\Sigma} {\bf 1}\{-1\le t\le  0\}\ \bar f_{-}(x) {\bf 1}_{B^{-}}(t,x,u)M(d(t,x,u)) ,
$$
hence
$$
\frac{1}{\ve}\bar \BP (A(\ve))\to \int_{\Sigma}  {\bf 1}_{B}(t,x,u){\mathbb Q}(d(t,x,u)) = {\mathbb Q}(B) ,
$$
since $B\subset \Sigma_1$. 
\qed

\bigskip\noindent
{\bf Remark.} In the theorem, we have assumed that the functions $\bar f_-, \bar f_+$ are integrable with respect to $|\Theta_{i}|(F,\cdot ),$ for $i=0,\dots,d-1$. For $i=0,..., d-2$, an easier condition, which is sufficient for \eqref{difmes}, is that these functions are integrable with respect to the measures $(r_{-}(\cdot)\wedge 1)^{d-i-1} |\Theta_{i}|(F,\cdot)$ and $(r_{+}(\cdot)\wedge 1)^{d-i-1} |\Theta_{i}|(F,\cdot)$, respectively.  This can be easily seen from the proof.

%However, even for bounded $\bar f_+$ and $\bar f_-$ thess conditions are not necessarily satisfied because the measures $(r_{-}(\cdot))^{j-1} |\Theta_{d-j}(F,\cdot))|$ and $(r_{+})^{j-1}(\cdot) |\Theta_{d-j}(F,\cdot))|$, and certainly the measures
%$|\Theta_{d-i}(F,\cdot)|$,  are not necessarily finite. It is useful to see which part of $\Nor(F)$ creates the problem and how it is to be rectified.

\section{Boundary sets}

As a second class of sets $F\subset \R$, we now study {\it boundary sets}. These are nonempty compact sets $F$ without interior points, hence $F=\dF$, and with $\mu_d(F)=0$. 
Again, notations and results can be generalized appropriately to unbounded closed sets $F$. 

Since $F^*=\R$, we need no extension of the normal bundle $\Nor (F)$ or the reach function $r$ and will use the Steiner formula \eqref{steiner1}. The regular points $x\in\dF$ can have one normal $\nu (F,x)$ (then $(x,-\nu (F,x))\notin \Nor (F)$) or two antipodal normals $\nu (F,x), -\nu(F,x)$ (here, we define $\nu(F,x)$ in some measurable way). The support measure $\Theta_{d-1}(F,\cdot)$ satisfies
\begin{align}\label{d-1meas2}
\Theta_{d-1}(F,\cdot) = \int_{\reg (F)} [{\bf1}\{ (x,\nu (F,x))\in\cdot\}+{\bf1}\{ (x,-\nu (F,x))\in\cdot\}]{\cal H}^{d-1}(dx) 
\end{align}
(see \cite[Prop. 4.1]{HLW}).
The {\it normal cylinder} $\Sigma = \Sigma(F)$ is then given by $\Sigma = {\mathbb R} \times \Nor (F)$. Note that the considerations in this section make sense for boundary sets $F$ with ${\cal H}^{d-1} (F)=0$ (e.g. for line segments in ${\mathbb R}^3$). Then $\reg (F) =\emptyset$ and $\Theta_{d-1}(F,\cdot) =0$, which implies that the following results are not very interesting for such sets.

The {\it local magnification map} $\tau_{\ve}$, 
$$      
\tau_{\ve}(z)  = (\frac{d(z)}{\ve},p(z),u(z)),
$$
is now defined for $z\in \R\setminus (S_F\cup F)$, and
is bicontinuous and one-to-one with image 
$$
\left\{(t,x,u) : (x,u) \in \Nor(F), t\in (0,\frac{r(x)}{\ve})\right\}\subset \Sigma.
$$

The further notations, definitions and results from Section 3 (up to and including Lemma 3) now carry over to our new situation either word-by-word or with the obvious changes. Since $F=\dF$, we now have only one notion of differentiability. Namely,  the set-valued mapping $F(\ve), 0\leq \ve \leq 1,$ (with $F(0)\subset F$) is called {\it differentiable} at $F$ for $\ve=0$, if it is essentially bounded, in the sense that
\begin{equation*}
\frac{1}{\ve} \mu_d (F(\ve)\cap (\R\setminus (F)_{\ve T})) \to 0\quad{\rm as\ }\ve\to 0,
\end{equation*}
for some $T>0$, and if there exists a Borel set $B \subset \Sigma$ (the {\it derivative}) such that $M(\tau_{\ve}(F(\ve))\triangle B) \to 0$, as $\ve\to 0$. Here, we can replace $F(\ve)$ by $\tilde F(\ve) =F(\ve)\setminus F$ (thus $\tilde F(0)=\emptyset$) without changing the derivative. Since the set $F$ has no interior normals, the derivative $B$ is automatically contained in the upper part $\Sigma^+$ of $\Sigma$.

It is important, for the understanding, to see the connection between the notion of differentiability considered in this section with the one of the previous section, in the case where $F=\dF$ is the boundary $F=\partial G$ of a solid set $G$. It is easily seen, that a family $F(\ve)$ which is differentiable at $F$ is then differentiable at $G$ and vice versa. The derivatives $B$ at $F$ and $C$ at $G$ are formally different, since $C$ sits in the cylinder $\Sigma (G)$ and may consist of two parts $C^+$ and $C^-$, whereas $B$ sits in the cylinder $\Sigma (F)$ and satisfies $B=B^+$. They can, however, be easily transformed into each other. Each point $(x,u)\in \Nor_e(G)$ is represented in $\Nor (F)$ by two points $(x,u)$ and $(x,-u)$. The half cylinder $\Sigma^+(G)$ is mapped to $\Sigma (F)$ by the identity map, $(s,x,u)\mapsto (s,x,u)$, $(x,u)\in \Nor_e (G), s\ge 0$. The half cylinder $\Sigma^-(G)$ is mapped to (a different part of) $\Sigma (F)$ by the reflection $(s,x,u)\mapsto (-s,x,-u)$, $(x,u)\in \Nor_e(G), s<0$. In this way, $B_1=C^+$ is already a subset of $\Sigma (F)$ whereas $C^-$ is mapped to a set $B_2\subset \Sigma (F)$. Then, we have $B=B_1\cup B_2$, and this is a disjoint union!

We now continue with a result corresponding to Theorem 4.

\medskip
Let  $\BP$ be an absolutely continuous measure on $\R$ with density $f\ge 0$. We assume that there is a bounded measurable function $\bar f \ge 0$ on $F$, such that
\begin{align}\label{lambda2}
&\frac{1}{\ve}\int_{\R} {\bf 1}\{ d(F,z)\le\ve\}|f(z) - \bar f (p_{F}(z))| \mu_d(dz)  
\to 0 , 
\end{align}
as $\ve\to 0$, and define the measure $\BQ$ on $\Sigma$ by 
\begin{align*}%\label{Q2}
\BQ(d(s,x,u)) &=  ds \times \bar f(x)\Theta_{d-1}(F,d(x,u)) .
\end{align*}

\begin{thm}\label{thm:mes3}  Suppose that the measure $\BP$ satisfies condition \eqref{lambda2} and suppose that the function $\bar f$ is integrable with respect to $|\Theta_{i}|(F,\cdot ), i=0,\dots,d-1$. Let also $F(\ve)\subset (F)_{\ve T}$ (for some $T>0$) be a set-valued mapping which is differentiable at $F$ (with derivative $B\subset\Sigma_T$). 
Then
$$
\frac{d}{d\ve}\BP(F({\ve}))|_{\ve=0} = \BQ(\frac{d}{d\ve}F({\ve}) |_{\ve = 0}) = \BQ(B) .
$$
\end{thm}

\begin{proof} We may assume that $T=1$.

Since $\BP(F({0}))=0$ (due to our assumption $\mu_d(F)=0$), we have to establish the  asymptotic behaviour of  $\ve^{-1}\BP(F({\ve}))$. 

Again, we consider the auxiliary measure $\bar \BP$ on $F_{\ve}$ with density $z\mapsto\bar f (p_{F}(z))$. Condition \eqref{lambda2} implies that $\ve^{-1}[\BP (F(\ve)) - \bar \BP (F(\ve))] \to 0$,  hence we can concentrate  on $\ve^{-1}\bar \BP (F(\ve))$. 

Since 
$$
\bar \BP (F(\ve))  = \int_{\R} \bar f(p_{F}(z)){\bf 1}_{F(\ve)}(z)\, \mu_d(dz)
$$
and $z\mapsto \bar f(p_{F}(z)){\bf 1}_{F(\ve)}(z)$ is bounded with compact support, we can apply the local Steiner formula \eqref{steiner1}.  It follows that 
\begin{align}\label{asymp2}
&\frac{1}{\ve}\bar \BP (F(\ve))  = \frac{1}{\ve}\int_{\Nor(F)}\int_{0}^{r(x,u)\wedge\ve} \bar f(x) {\bf 1}_{F(\ve)}(x+tu) dt \, \Theta_{d-1}(d(x,u))\\
&\ + \sum_{j=2}^{d} \frac{1}{\ve}\binom{d-1}{j-1} \int_{\Nor(F)}\int_{0}^{r(x,u)\wedge\ve} \bar f(x) {\bf 1}_{F(\ve)}(x+tu) t^{j-1} dt \, \Theta_{d-j}(d(x,u)).\nonumber
\end{align}
Again, the sum of the higher order terms vanishes asymptotically, since
\begin{align*}
&\left|\int_{\Nor(F)}\int_{0}^{r(x,u)\wedge\ve} \bar f (x) {\bf 1}_{F(\ve)}(x+tu) t^{j-1} dt \, \Theta_{d-j}(d(x,u))\right|\cr
&\quad\leq \int_{\Nor(F)}\bar f(x)\left(\int_{0}^{r(x,u)\wedge\ve}  {\bf 1}_{F(\ve)}(x+tu) t^{j-1} dt\right) \, |\Theta_{d-j}|(d(x,u))\cr
&\quad\le \frac{\ve^{j}}{j}\int_{\Nor(F)}\bar f(x)\, |\Theta_{d-j}|(d(x,u))
\end{align*}
with $j\ge 2$, 
and the latter integral is finite, by our assumptions.

For the first summand in \eqref{asymp2}, we have
\begin{align*}
\frac{1}{\ve}&\int_{\Nor(F)}\int_{0}^{r(x,u)\wedge\ve} \bar f(x) {\bf 1}_{F(\ve)}(x+tu) dt \, \Theta_{d-1}(d(x,u))\cr
 &=  \int_{\Sigma} {\bf 1}\{0\le t\le \frac{r(x,u)}{\ve}\wedge 1\}\ \bar f(x) {\bf 1}_{B(\ve)}(t,x,u)M(d(t,x,u)) 
\end{align*}
with $B(\ve) = \tau_\ve (F(\ve)\setminus F)$.

Since the function
$|{\bf 1}_{B(\ve)}(t,x,u) - {\bf 1}_{B}(t,x,u)|$ tends to $0$ $M-$a.e. on $\Sigma$, Lebesgue's theorem of majorised convergence implies that
$$
\int_{\Sigma} {\bf 1}\{0\le t\le \frac{r(x,u)}{\ve}\wedge 1\}\ \bar f (x) |{\bf 1}_{B (\ve)}(t,x,u) - {\bf 1}_{B}(t,x,u)| M(d(t,x,u))
$$
tends to $0$, as $\ve\to 0$.
This shows that
$$
\frac{1}{\ve}\bar \BP (F(\ve))\to \int_{\Sigma} {\bf 1}\{0\le t\le  1\}\ \bar f(x) {\bf 1}_{B}(t,x,u)M(d(t,x,u)) ,
$$
hence
$$
\frac{1}{\ve}\bar \BP (F(\ve))\to \int_{\Sigma}  {\bf 1}_{B}(t,x,u){\mathbb Q}(d(t,x,u)) = {\mathbb Q}(B) ,
$$
since $B\subset \Sigma_1$. 
\end{proof}

\bigskip\noindent
{\bf Remark.} Similarly as in the last section (see the remark after Theorem \ref{thm:mes}), the integrability conditions on $\bar f$ can be relaxed.

\section{Set functions with bifurcation}

Motivated by possible applications, we now consider a special situation of a family $F(\ve),0\le\ve\le 1$, which is a finite union 
$$
F(\ve) = \bigcup_{i=1}^N F_i(\ve)
$$
of families $F_i(\ve), 0\le\ve\le 1,$ of compact sets which, for $\ve>0$, are pairwise disjoint, that is $ F_i(\ve)\cap F_j(\ve)=\emptyset$, if $i\not= j$. Assume that the sets $F_i=F_i(0)$ are solid and that their interiors are pairwise disjoint. It is then easy to see, that
$$
F=F(0)=\bigcup_{i=1}^N{F_i}$$
is a solid set. If the sets $F_i$ themselves  are pairwise disjoint, then we can consider the families $F_i(\ve)$ individually and are back in the situation of Section 3. The more interesting situation occurs, if  there are non-empty boundary parts  $C_i = \dF_i\setminus \dF$ of $F_i$ in $F=F(0)=\bigcup_{i=1}^NF_i$. These sets $C_i$ may then be interpreted as bifurcation surfaces (or cracks) which arise in $F$ as a result of the evolution in $\ve$. Notice that each point $x\in C_i$ also lies in $C_j$, for some $j\not= i$ (or even in more than two sets $C_i$). We put $C_{ij}=C_i\cap C_j$, for $i\not= j$. Our boundary set of interest is then
$$
C = \bigcup_{i=1}^N \dF_i = \dF \cup \bigcup_{1\le i <j\le N} C_{ij} .
$$

Let us assume now that $F_i(\ve)$ is differentiable at $F_i$ with derivative $B_i$, for $i=1,...,N$. Is then $F(\ve)\triangle F$ differentiable at $C$? And, if ``yes'', what is the derivative? The following example (for $N=2$) shows that we cannot expect a positive answer without further assumptions. In the example, we have $F=F_1\cup F_2$ and $\dF=\dF_1\cup\dF_2$, thus $C_{12}=\emptyset$ which makes the calculation simpler. A corresponding example with $C_{12}\not=\emptyset$ can be easily obtained by adding sets $\tilde F_1,\tilde F_2$ to $F_1,F_2$, disjoint from $F(\ve)$ and such that the corresponding set $\tilde C_{12}$ is nonempty.\\

\noindent
{\bf Example.} Let  $a_k, k=1,2,\dots,$ be a monotone sequence, which converges to zero, and let $b_k=(a_k+a_{k+1})/2$. Consider
$$F_1 = \left(\bigcup_{k=1}^\infty \left[b_k, a_k \right]\cup \{0\}\right)\times [0,1] ,\ F_2 = [-1,0]\times [0,1], $$
both as subsets of ${\mathbb R}^2$. Both sets, $F_1$ and $F_2$, are solid and the joint boundary part $\dF_1\cap\dF_2$ is the segment $S = \{0\}\times [0,1]$.
Let $F_1(\ve)=F_1$ and $F_2(\ve)=[-1,\ve]\times[0,1]$. Then, $F_1(\ve)$ is differentiable at $F_1$ with derivative $\emptyset$ and $F_2(\ve)$ is differentiable at $F_2$ with derivative $$B=\{(t,x,u)\in\Sigma (F_2) : 0\le t\le 1, x\in S, u=(1,0)\} .$$

\begin{figure} %[!htbp] 
\begin{center}
\includegraphics[scale=0.65]{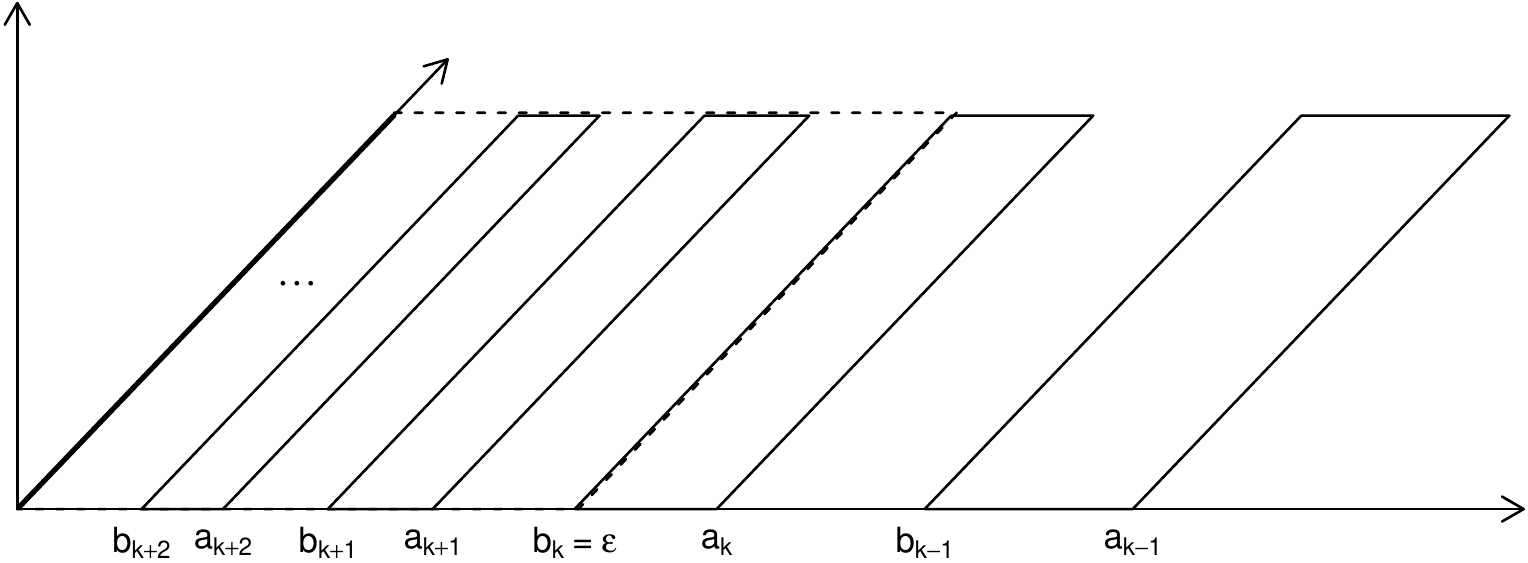}
\caption{The set $F_1$ along with $A(\ve)=(0,\ve]\times[0,1]$ (dotted line)}
\label{fig:SD1} 
\end{center} 
\end{figure}

We show that $A(\ve) = F(\ve)\triangle F = (0,\ve]\times[0,1]\setminus F_1$ is not differentiable at $\dF$. In fact, all points of $A(\ve)$ project onto $\dF_1$ and therefore, with respect to $A(\ve)$, the local magnification map $\tau_{\ve}$ of $C=\dF$ is the same as the local magnification map $\tau_{\ve}^{(1)}$ of $\dF_1$.  
%Denote $u_0$ the normal, which, in the unit sphere, has coordinates $(1,0)$. 
For $\ve = b_k$ and $u_0=(0,1)$, we therefore get
\begin{align*}
B(\ve) &= \tau_{\ve}(A(\ve)) = \tau_{\ve}^{(1)}((0,\ve]\times[0,1]\setminus F_1)\cr
&=\bigcup_{i=k}^\infty\left\{ (t,x,u) : t\in (0,\frac{b_i-a_{i+1}}{2 b_i}), x\in \{a_{i+1}\}\times [0,1], u=u_0\right\}\cr
&\cup\bigcup_{i=k}^\infty\left\{ (t,x,u) : t\in (0,\frac{a_i-b_i}{2b_i}), x\in \{b_i\}\times [0,1], u=-u_0\right\}.
{\color{red} /2b_k}
\end{align*}

\begin{figure} %[!htbp] 
\begin{center}
\includegraphics[scale=0.65]{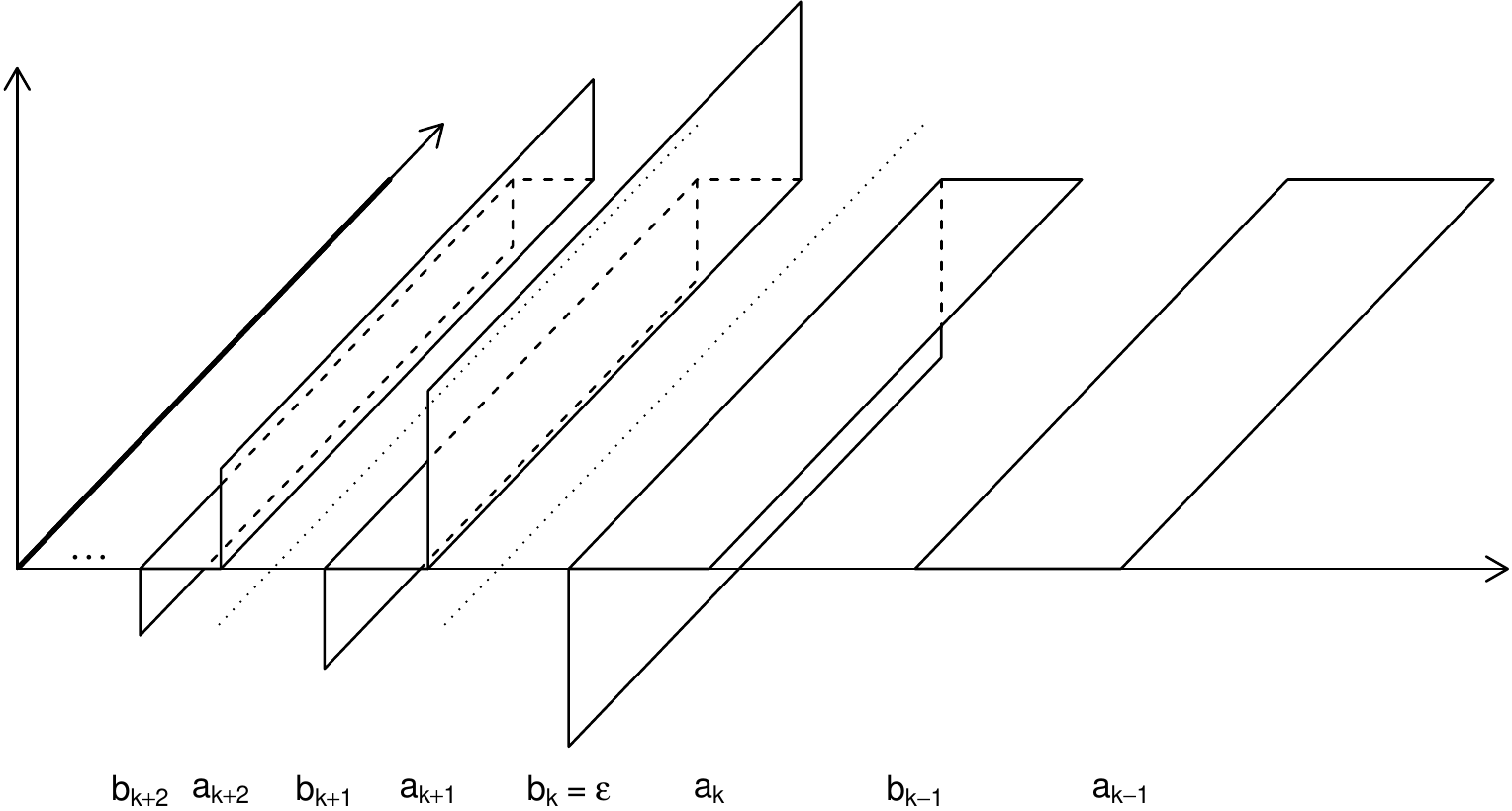}
\caption{An illustration of the form of $B(\ve)$}
\label{fig:SD2} 
\end{center} 
\end{figure}

The measure $M$ of this set remains strictly positive,
$$
M(B(\ve)) = \sum_{i=k}^\infty \frac{a_i-a_{i+1}}{2b_k} =\frac{a_k}{2b_k} = \frac{a_k}{a_k+a_{k+1}} \in [1/2, 1],
$$
whereas the set $B(\ve)$ itself shrinks to a subset of $B_0 = \left([0,1/4]\times S\times \{u_0\}\right)\cup  \left( [0,1/4]\times S\times \{-u_0\}\right)$. Notice that $B_0\cap\Sigma$ is empty, since $u_0$ and $-u_0$ are not normals of $F$ at $x\in S$. Hence, there cannot be a set $B\subset\Sigma$ with $M(B(\ve)\triangle B)\to 0$ and therefore $A(\ve)$ is not differentiable at $\dF$.

\bigskip
The additional restrictions, which we have to impose on the sets $F_i, i=1,...,N,$ and the proof of the differentiability result for the union set $F=\bigcup_{i=1}^NF_i$ becomes a bit technical, for general $N$. We therefore concentrate now on the case $N=2$, but the general case can be treated in a similar way. \\

\noindent
{\bf Definition 3.} Let $F_1,F_2$ be solids sets. We say that $F_1,F_2$ provide a {\it normal} decomposition of the solid set $F=F_1\cup F_2$, if $F_1$ and $F_2$ have only boundary points in common and if
\begin{equation}\label{newcond}
\frac{1}{\ve}\mu_d\left( (\dF)_{\ve}\cap (\dF_1)_\ve \cap (\dF_2)_{\ve}\right)\to 0
\end{equation}
holds, as $\ve\to 0$.

\begin{figure} %[!htbp] 
\begin{center}
\includegraphics[scale=0.65]{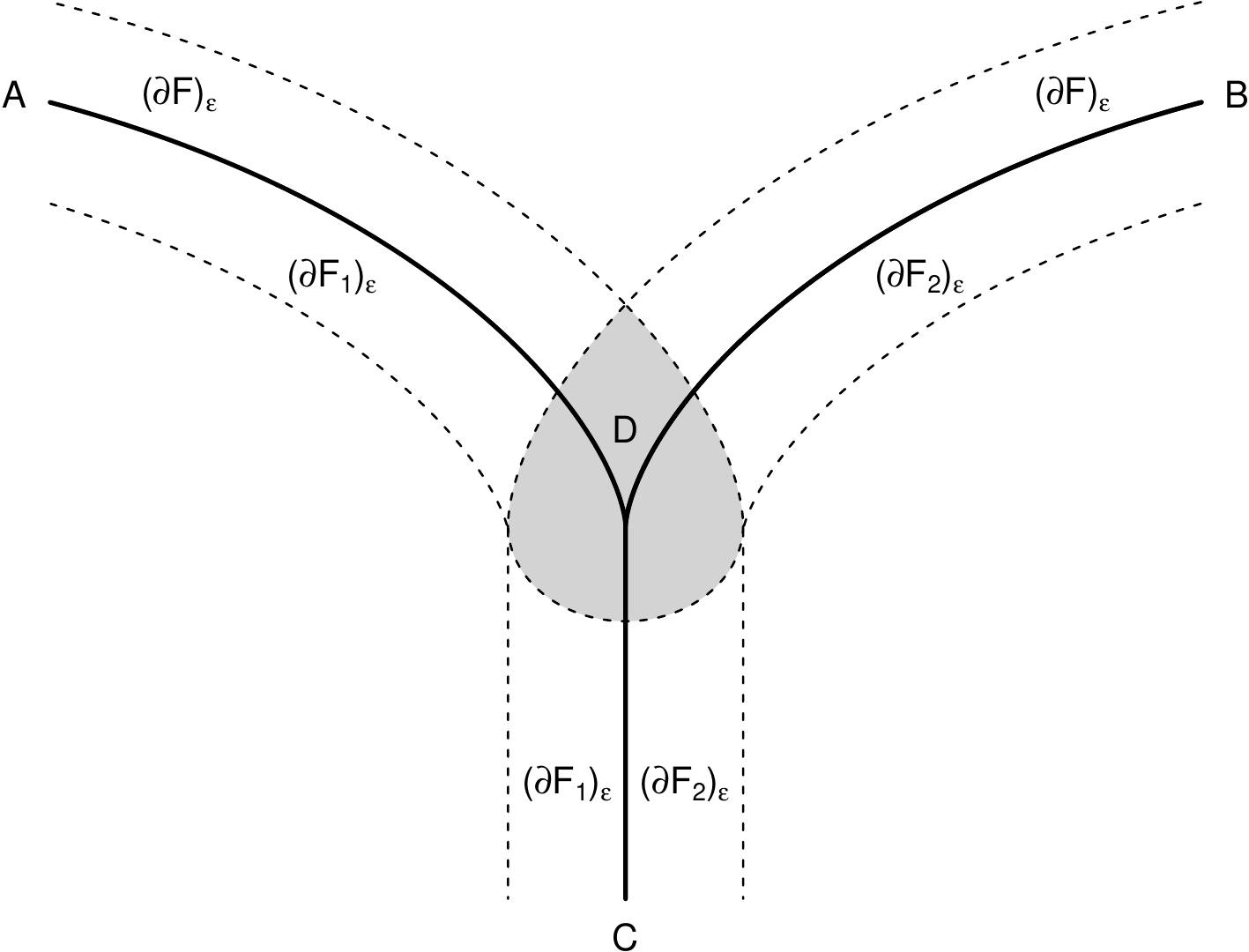}
\caption{The shaded set illustrates condition \eqref{newcond}. The curve ADC is part of $\dF_1$, BDC is part of $\dF_2$, while ADB is part of $\dF$.}
\label{fig:EN} 
\end{center} 
\end{figure}

\bigskip
\begin{lem} Suppose $F_1,F_2$ yield a normal decomposition of $F=F_1\cup F_2$. Then  
\begin{equation}\label{boundpoints}
\Theta_{d-1} ( C,\{ (x,u)\in \Nor (C) : x\in \dF\cap \dF_1\cap\dF_2 \} )= 0.
\end{equation}
Moreover, we have
\begin{equation}\label{boundcond}
\frac{1}{\ve}\mu_d\left( (\dF_1\triangle\dF_2)_{\ve}\cap (\dF_1\cap \dF_2)_{\ve}\right)\to 0
\end{equation}
and
\begin{equation}\label{boundcond2}
\frac{1}{\ve}\mu_d\left( ((\dF_1)_\ve\cap(\dF_2)_{\ve})\setminus (\dF_1\cap \dF_2)_{\ve}\right)\to 0,
\end{equation}
as $\ve\to 0$.
\end{lem}

\begin{proof} The first assertion, \eqref{boundpoints}, is a direct consequence of \eqref{newcond} and the Steiner formula. 

Since $\dF_1\triangle\dF_2\subseteq \dF,$ we have $(\dF_1\triangle\dF_2)_{\ve}\subseteq (\dF)_{\ve}$ and at the same time $$(\dF_1\cap \dF_2)_{\ve} \subseteq (\dF_1)_\ve\cap(\dF_2)_{\ve} .$$
Therefore, 
$$
(\dF_1\triangle\dF_2)_{\ve}\cap (\dF_1\cap \dF_2)_{\ve}\subseteq (\dF)_{\ve}\cap (\dF_1)_\ve \cap (\dF_2)_{\ve} ,
$$
and \eqref{boundcond} follows from condition \eqref{newcond}. 

With respect to \eqref{boundcond2}, 
$$
z\in ((\dF_1)_\ve\cap(\dF_2)_{\ve})\setminus (\dF_1\cap \dF_2)_{\ve}
$$
implies that the distances $d_{\dF_1}(z), d_{\dF_2}(z)$ do not exceed  $\ve$, but $d_{\dF_1\cap\dF_2}(z) > \ve$. Therefore $z$ is on a distance smaller than or equal $\ve$ not from $\dF_1\cap\dF_2$, but from $\dF_1\triangle \dF_2$, that is, from $\dF$. Hence
$$
((\dF_1)_\ve\cap(\dF_2)_{\ve})\setminus (\dF_1\cap \dF_2)_{\ve}\subseteq (\dF)_{\ve}\cap (\dF_1)_\ve \cap (\dF_2)_{\ve}
$$
and \eqref{boundcond2} follows, again from \eqref{newcond}.
\end{proof}

We now discuss the normal cylinders of the sets involved. Apparently, this reduces to a discussion of the corresponding normal bundles.
The normal bundle $\Nor(C)$ of $C$ can be embedded as a subset into the union of the normal bundles $\Nor (\partial F_i)$, $i=1,2$. In fact, any $(x,u)\in \Nor (C)$ comes from a point $z\notin C$ which (uniquely) projects onto $x\in C$. If $x\in \dF_1$, then $x$ is also the projection of $z$ onto $\dF_1$ and hence $(x,u)\in\Nor (F_1)$. We similarly argue if $x\in\dF_2$.

In order to embed also $\Nor (\dF_1)$ into $\Nor (C)$, we have to neglect pairs $(x_1,u)$ from $\Nor (\dF_1)$ for which $u$ is not a normal at $x$ in $C$. For such a pair $(x_1,u)\in \Nor (\dF_1)$, there exists small enough $\ve$, such that all $z=x_1+t u, t\leq \ve,$ project onto $x_1$. Since $\dF_1\subseteq C$, there is a point $x_2\in C$ with
$$\|z-x_2\|=\inf_{x\in C}\|z-x\| \leq \inf_{x\in \dF_1}\|z-x\|=\|z-x_1\|,$$
and therefore all points $z=x_1+t u, t\leq \ve,$ are in $C_\ve$. However, $x_2$ has to be different from $x_1$. Otherwise, we would have $(z-x_2)/\|z-x_2\| = u$ and $(x_1,u)\in \Nor (C)$, a contradiction. Therefore, we obtain $z\in (\dF_1)_\ve\cap (\dF_2)_\ve \setminus C_\ve$. Applying the local magnification map $\tau_\ve^{(1)}$ to this set and using \eqref{boundcond2} and the Steiner formula, we deduce that
$$
\Theta_{d-1}(\dF_1,\Nor (\dF_1)\setminus \Nor (C)) =0,
$$
which shows that we can embed $\Sigma(\dF_1)$ into $\Sigma (C)$, up to a set of measure 0. In the same way, we can embed $\Nor (\dF_2)$ into $\Nor (C)$.

\medskip
Therefore, we identify now the normal cylinders $\Sigma (C)$ and $\Sigma (F_1) \cup\Sigma (F_2)$. If $B_i$ denotes the derivative  of $F_i(\ve)$ at $F_i$, $i=1,2$, 
the positive part $B_i^+$ and the reflection $R(B_i^-)$ of its negative part $B_i^-$ can be seen as subsets of $\Sigma (\partial F_i)$, as we have explained before Theorem \ref{thm:mes3} and therefore also as subsets of $\Sigma (C)$, for $i=1,2$. Since we can also embed the normal cylinders $\Sigma(F)$ of $F$ and $\Sigma(F_i)$ of $F_i$, $i=1,2$,  into $\Sigma = \Sigma (C)$ by the mapping $(t,x,u)\mapsto (t,x,u)$, for $t\ge 0$, and by the reflection $R: (t,x,u)\mapsto (-t,x,-u)$, for $t<0$, we can extend the measures $M_F$ and $M_{F_i}, i=1,2$, to $\Sigma (C)$, in the obvious way. We denote by $M_F^+, M_F^-$ the restrictions of $M_F$ to $\Sigma^+(F)$ respectively $\Sigma^-(F)$ and we use similar notations for the measures $M_{F_i}$, $i=1,2$.

\begin{lem}\label{lemma}
Suppose $F_1,F_2$ yield a normal decomposition of $F=F_1\cup F_2$. Then 
\begin{equation}\label{MCmeasure}
M_C = M_F^+ + M_{F_1}^-\circ R + M_{F_2}^-\circ R.
\end{equation}
\end{lem}

\begin{proof} 
The assertion follows from a corresponding decomposition of the support measures,
\begin{align*}
\Theta_{d-1}(C,\cdot) &=  \Theta_{d-1}(F,\cdot) +\Theta_{d-1}(F_1^\ast,\cdot)+\Theta_{d-1}(F_2^\ast,\cdot)
\end{align*}
which is a consequence of \eqref{d-1meas1} and \eqref{d-1meas2}, together with \eqref{boundpoints}.
\end{proof}

We now formulate our main result in this section.

\begin{thm}\label{cracks} Let $F_1(\ve), F_2(\ve), 0\le\ve\le 1,$ be two families of nonempty compact sets such that, for each fixed $\ve >0$, the sets $F_1 (\ve)$ and $F_2(\ve)$ are pairwise disjoint, and let $F(\ve) = F_1(\ve)\cup F_2(\ve)$. Assume that $F_1=F_1(0)$ and $F_2=F_2(0)$ provide a normal decomposition of $ F = F(0)= F_1\cup F_2$. 

If the families $F_i(\ve)$ are differentiable at $F_i$ with derivative $B_i$, $i=1,2$, then
$$
C(\ve) = F(\ve)\triangle F
$$
is  differentiable at
$$
C=  \dF_1\cup\dF_2  
$$
with derivative
$
B= \tilde B_1\cup\tilde B_2 
$
where
$$
\tilde B_i =\left\{ {R(B_i^-)\setminus B_j^+ {\text\ on \ } \Sigma (C_{ij}), {\text\ for\ }j\not= i,}\atop {B_i^+\cup R(B_i^-){\text{ \ otherwise.}}}\right.
$$
\end{thm}

\begin{proof} 
We start with the essential boundedness condition.
Since the families $F_i(\ve)\triangle F_i$ are essentially bounded, we may assume that there is a $T$  such that
$$
F_i(\ve)\triangle F_i \subset (\dF_i)_{T\ve}, \quad i=1,2.
$$
We may also put $T=1$. It is then easy to see that
$$
F(\ve)\triangle F \subset C_{\ve} ,
$$
hence $C(\ve)$ is essentially bounded.

In order to show that $F(\ve )\triangle F$ is differentiable at $C$ with derivative $B$, it remains to show that
$$M_C(\tau_\ve (F(\ve)\triangle F)\triangle B) \to 0,$$
as $\ve\to 0$. 
Observe that here $\tau_\ve$ is the magnification map belonging to $C$. Later, we will also use the magnification map $\tau_\ve^{(i)}$ belonging to $\dF_i$, $i=1,2$. For $z\notin C\cup S_C\cup S_{\partial F_1}\cup S_{\partial F_2}$, we then have $\tau_\ve (z) =  \tau_\ve^{(i)} (z)$, for some $i$ (possibly for both). 
We now use \eqref{MCmeasure} and discuss the effects of the different summands of $M_C$ to the set $\tau_\ve (F(\ve)\triangle F)\triangle B$ separately. 

Since $M_F^+$ is concentrated on $[0,\infty)\times \Nor (F)$ (notice that we can use $\Nor (F)$ instead of $\Nor_e(F)$ here), we can decompose $M_F^+$ into a sum $$M_F^+ = M^{(1)}+ M^{(2)},$$ where
$$
M^{(i)} = \mu_1^+\otimes \left[\Theta_{d-1}(F_i,\cdot)\fed \left(\Nor(F_i) \cap \Nor (F)\right)\right],\ i=1,2.
$$
Here, $\mu_1^+$ is the Lebesgue measure on $[0,\infty)$ and $\rho\fed A$ denotes the restriction of the measure $\rho$ to the set $A$. Using this decomposition and the facts that 
$$
\tau_\ve (F(\ve)\triangle F)\triangle B \subset [\tau_\ve (F(\ve)\setminus F)\triangle B]\cup \tau_\ve (F)
$$
and $M_F^+(\tau_\ve (F))=0,$ we first obtain
\begin{align*}
M_F^+(\tau_\ve (F(\ve)\triangle F)\triangle B)
&\le M_F^+(\tau_\ve (F(\ve)\setminus F)\triangle B)\cr
&=\sum_{i=1}^2 M^{(i)}(\tau_\ve (F(\ve)\setminus F)\triangle B).
\end{align*}
On $[0,\infty)\times (\Nor (F_1)\cap \Nor (F))$, we have 
\begin{align*}
\tau_\ve (F(\ve)\setminus F) &=\tau_\ve^{(1)} ((F(\ve)\setminus F)\cap (\dF_1)_\ve)\cr
&= \tau_\ve^{(1)}(F_1(\ve)\setminus F_1) \cup \tau_\ve^{(1)}((F_2(\ve)\setminus F_1)\cap (\dF_1)_\ve),
\end{align*}  
hence
\begin{align}\label{ineq1}
M^{(1)}&(\tau_\ve (F(\ve)\setminus F)\triangle B)\cr &\le M^{(1)}(\tau_\ve^{(1)}(F_1(\ve)\setminus F_1)\triangle B) +M^{(1)}(\tau_\ve^{(1)}((F_2(\ve)\setminus F_1)\cap (\dF_1)_\ve)).\quad\quad
\end{align}
Moreover,
$$
M^{(1)}(B_2^+) = M^{(1)}(R(B_2^-)) = M^{(1)}(B_1^-)= M^{(1)}(R(B_1^-))=0$$
(the latter fact arises, since $\Nor(F_1)$ and $R(B_1^-)$ are disjoint subsets of $\Nor (C)$).
Therefore,
\begin{align*}
M^{(1)}(\tau_\ve^{(1)} (F_1(\ve)\setminus F_1)\triangle B)&=
M^{(1)}(\tau_\ve^{(1)} (F_1(\ve)\setminus F_1)\triangle B_1^+)\cr
&= M^{(1)}(\tau_\ve^{(1)} (F_1(\ve)\triangle F_1)\triangle B_1^+)\cr
&= M^{(1)}(\tau_\ve^{(1)} (F_1(\ve)\triangle F_1)\triangle B_1).
\end{align*}
Since $M^{(1)}\le M_{F_1}$, we get
\begin{equation}\label{ineq2}
M^{(1)}(\tau_\ve^{(1)} (F_1(\ve)\setminus F_1)\triangle B)\le M_{F_1}(\tau_\ve^{(1)} (F_1(\ve)\triangle F_1)\triangle B_1)\to 0 ,
\end{equation}
as $\ve\to 0$, due to the differentiability of $F_1(\ve)$. 

Furthermore, we notice that points $z$ in $(F_2(\ve)\setminus F_1)\cap (\dF_1)_\ve$ which project onto $\dF_1$ must lie in $(\dF_1\setminus \dF_2)_\ve \cap (\dF_1\cap\dF_2)_\ve$, and so
\begin{align*}
M^{(1)}&(\tau_\ve^{(1)}((F_2(\ve)\setminus F_1)\cap (\dF_1)_\ve))\cr
&\le M^{(1)}(\tau_\ve^{(1)}((\dF_1\setminus \dF_2)_\ve \cap (\dF_1\cap\dF_2)_\ve)).
\end{align*}
The local Steiner formula \eqref{steiner1} shows that
\begin{align*}
\frac{1}{\ve} &\mu_d((\dF_1\setminus \dF_2)_\ve \cap (\dF_1\cap\dF_2)_\ve)\cr
&\quad =  M^{(1)}(\tau_\ve^{(1)}((\dF_1\setminus \dF_2)_\ve \cap (\dF_1\cap\dF_2)_\ve)) + o(\ve).
\end{align*}
Therefore, \eqref{boundcond} implies
\begin{equation}\label{ineq3}
 M^{(1)}(\tau_\ve^{(1)}((\dF_1\setminus \dF_2)_\ve \cap (\dF_1\cap\dF_2)_\ve)) \to 0,
\end{equation}
as $\ve\to 0$. Combining \eqref{ineq1}, \eqref{ineq2} and \eqref{ineq3} gives 
\begin{equation*}
 M^{(1)}(\tau_\ve (F(\ve)\setminus F)\triangle B) \to 0.
\end{equation*}

In the same way, we get
\begin{equation*}
 M^{(2)}(\tau_\ve (F(\ve)\setminus F)\triangle B) \to 0,
\end{equation*}
hence
\begin{equation}\label{firstpart}
 M_F^+(\tau_\ve (F(\ve)\triangle F)\triangle B) \to 0.
\end{equation}

Now, we consider 
$$
(M_{F_1}^-\circ R)(\tau_\ve (F(\ve)\triangle F)\triangle B).
$$
Observe that $\tilde M_{F_1}= M_{F_1}^-\circ R$ is a measure on $R(\Sigma^-(F_1))$. On this set, we have
\begin{align*}
\tau_\ve (F(\ve)\triangle F) &= \tau_\ve^{(1)} (F_1\setminus F(\ve))\cr
&= [\tau_\ve^{(1)} (F_1\setminus F(\ve))\cap \Sigma (F)] \cup [\tau_\ve^{(1)} (F_1\setminus F(\ve))\cap \Sigma_{12}],
\end{align*}
here $\Sigma_{12}= [0,\infty)\times (\Nor (\dF_1)\setminus \Nor (F))$ is the normal cylinder of $C_{12}$ in relative interior points of $C_{12}$ and with normals $u$ pointing into the interior of $F_1$. Notice that the sets $\tau_\ve^{(1)} ((F_1\setminus F(\ve))\cap \Sigma (F))$ and $\tau_\ve^{(1)} ((F_1\setminus F(\ve))\cap \Sigma_{12})$ live on different parts of the cylinder $\Sigma (C)$. Therefore,
\begin{align}\label{sum}
\tilde M_{F_1}&(\tau_\ve (F(\ve)\triangle F)\triangle B)
= \tilde M_{F_1}((\tau_\ve^{(1)} (F_1\setminus F(\ve))\cap \Sigma (F))\triangle (R(B_1^-)\cap\Sigma(F)))\nonumber\\
&\quad + \tilde M_{F_1}((\tau_\ve^{(1)} (F_1\setminus F(\ve))\cap \Sigma_{12})\triangle( (B_2^+\cup R(B_1^-))\cap\Sigma_{12}))\nonumber\\
&= \tilde M^{(1)}(\tau_\ve^{(1)} (F_1\setminus F(\ve))\triangle R(B_1^-))\\
&\quad + \tilde M^{(2)}(\tau_\ve^{(1)} (F_1\setminus F(\ve))\triangle (B_2^+\cup R(B_1^-))).\nonumber
\end{align}
Here, $\tilde M^{(1)}$ denotes the restriction of $\tilde M_{F_1}$ to $\Sigma (F)$ and $\tilde M^{(2)}$ is the restriction to $\Sigma_{12}$.

For the first summand, we use
\begin{align*}
\tau_\ve^{(1)} (F_1\setminus F(\ve))
=\tau_\ve^{(1)} (F_1\setminus F_1(\ve))
 \setminus\tau_\ve^{(1)} (F_1\cap F_2(\ve))
\end{align*}
and 
\begin{align*}
\tilde M^{(1)}(\tau_\ve^{(1)} &(F_1\setminus F_1(\ve))\triangle R(B_1^-))\to 0,
\end{align*}
since $F_1(\ve)$ is differentiable at $F_1$. Also
\begin{align}\label{limit0}
\tilde M^{(1)}(\tau_\ve^{(1)} (F_1\cap F_2(\ve)))\to 0.
\end{align}
In fact, the Steiner formula \eqref{steiner1} shows that
\begin{align*}
\tilde M^{(1)}(\tau_\ve^{(1)} (F_1\cap F_2(\ve))) =\frac{1}{\ve} \mu_d ((\dF_1)_\ve\cap (F_1\cap F_2(\ve))) + o(\ve).
\end{align*}
 Points $z\in (\dF_1)_\ve\cap (F_1\cap F_2(\ve))$ which project onto $\dF$ lie in $(\dF_1)_\ve\cap(\dF_2)_\ve\cap (\dF)_\ve$. Hence, the assertion follows from \eqref{newcond}.
  
 Together we get
\begin{align}\label{firstterm}
\tilde M^{(1)}(\tau_\ve^{(1)} (F_1\setminus F(\ve))\triangle R(B_1^-)) \to 0.
\end{align}

For the second summand in \eqref{sum}, we similarly have
\begin{align*}
\tau_\ve^{(1)} (F_1\setminus F(\ve)) =\tau_\ve^{(1)} (F_1\setminus F_1(\ve)) 
\setminus \tau_\ve^{(1)} (F_1\cap F_2(\ve))
\end{align*}
with
$$
\tilde M^{(2)}(\tau_\ve^{(1)} (F_1\setminus F_1(\ve))\triangle R(B_1^-))\to 0,$$
 again by the differentiability of $F_1(\ve)$. On the other hand, 
\begin{align*}
&\tilde M^{(2)}(\tau_\ve^{(1)} (F_1\cap F_2(\ve)))=
\tilde M_{F_1}(\tau_\ve^{(1)} ((F_1\cap F_2(\ve))\cap \Sigma_{12}))\cr
&= M_{F_2}^+(\tau_\ve^{(2)} (F_1\cap F_2(\ve))\cap \Sigma_{12})
 -\tilde M_{F_1}(\tau_\ve^{(1)} (F_1\cap F_2(\ve))\cap \Sigma (F)),
\end{align*}
taking into account the points in $F_1\cap F_2(\ve)$ which project onto $\dF$ and not onto $C_{12}$. Here,
$$M_{F_2}^+((\tau_\ve^{(2)} (F_1\cap F_2(\ve))\cap \Sigma_{12})\triangle(B_2^+\cap \Sigma_{12}))\to 0,$$
 by the differentiability of $F_2(\ve)$ and the term 
$$ \tilde M_{F_1}(\tau_\ve^{(1)} (F_1\cap F_2(\ve))\cap \Sigma (F))
$$
converges to 0 by our condition \eqref{newcond}, as we have seen in \eqref{limit0}. Hence,
$$\tilde M^{(2)}(\tau_\ve^{(1)} (F_1\cap F_2(\ve))\triangle B_2^+)\to 0.$$ 

Together we obtain
\begin{align}\label{secondterm}
&\tilde M^{(2)}(\tau_\ve^{(1)} (F_1\setminus F(\ve))\triangle(R(B_1^-)\setminus B_2^+)) \cr
&\ =\tilde M^{(2)}(\tau_\ve^{(1)} (F_1) \setminus F_1(\ve))\triangle \tau_\ve^{(1)} (F_1\cap F_2(\ve))\triangle(R(B_1^-)\triangle B_2^+))\cr
&\ \le \tilde M^{(2)}(\tau_\ve^{(1)} (F_1) \setminus F_1(\ve))\triangle R(B_1^-))+\tilde M^{(2)}(\tau_\ve^{(1)} (F_1\cap F_2(\ve))\triangle B_2^+)\cr
&\ \to 0.
\end{align}
From \eqref{firstterm} and \eqref{secondterm}, we arrive at
\begin{align}\label{secondpart}
(M_{F_1}^-\circ R)&(\tau_\ve(F\triangle F(\ve)))\triangle B)\to 0.
\end{align}

In the same manner, we get
\begin{align}\label{thirdpart}
(M_{F_2}^-\circ R)&(\tau_\ve(F\triangle F(\ve)))\triangle B)\to 0.
\end{align}
Combining \eqref{firstpart}, \eqref{secondpart} and \eqref{thirdpart}, we obtain the asserted differentiability.
\end{proof} 

\section{Parallel sets}

We now discuss some particular classes of  set-valued mappings which are differentiable, the subgraphs and the local or global parallel sets. 

Let $F=F(0)$ be a solid set and $h_{\ve}, 0\le \ve \le 1$, a family of nonnegative measurable functions on $\Nor (F)$ (with $h_0 =0$). As in \cite{Khm07}, we call
\begin{align*}
h_{\ve,{\rm sub}} &=\{ z=x+tu : (x,u)\in \Nor (F), 0< t\le h_\ve(x,u)\wedge r(x,u)\}
\end{align*}
the {\it subgraph} of $h_\ve$. We assume that the following two conditions hold:
\begin{enumerate}[label={(\alph*)}] 
\item{} For each $(x,u)\in \Nor (F)$, $\ve\to h_\ve (x,u)$ is differentiable at $\ve=0$  with derivative $g(x,u)$. Thus
$$
\frac{h_\ve(x,u)}{\ve}\to g(x,u), \quad
\ve\to 0.
$$
\item{}
There is a $\delta >0$, such that the function  $\max_{0<\ve\le\delta} \frac{h_\ve}{\ve}$ is  bounded and integrable  with respect to $\Theta_{d-1}(F,\cdot)$. Hence,
\begin{equation}\label{upperbound}
 \max_{0<\ve\le\delta} \frac{h_\ve(x,u)}{\ve}\le T,
\end{equation}
for some $T>0$ and 
$$\int_{\Nor (F)} \max_{0<\ve\le\delta} \frac{h_\ve(x,u)}{\ve} \, \Theta_{d-1}(F,d(x,u)) <\infty .
$$
\end{enumerate}

\begin{thm}\label{subgraphs}
Let $F$ be solid and let $h_{\ve}, 0\le \ve \le 1$, be a family of nonnegative measurable functions on $\Nor (F)$ satisfying conditions (a) and (b). Then, $A(\ve) = h_{\ve,{\rm sub}}$ is differentiable at $\dF$ and the derivative is
$$
B= \{(t,x,u) : 0< t\le g(x,u), (x,u)\in\Nor (F)\}.
$$
\end{thm}

\begin{proof}
We first show that $A(\ve)$ is essentially bounded. Let $\delta$ be given as in (b) and let $T$ be the bound from \eqref{upperbound}. Suppose $\ve\le\delta$. Then, 
\begin{align*}
&\frac{1}{\ve} \mu_d (A(\ve)\cap (\R\setminus (\dF)_{\ve T})) \cr
&\ = \frac{1}{\ve} \mu_d (\{ z=x+tu : (x,u)\in\Nor (F), \ve T< t\le h_\ve(x,u)\wedge r(x,u)\}\cr
\end{align*}
equals $0$, because condition (b) implies that the set here is empty.
Hence, $A(\ve)$ is essentially bounded. 

With respect to the differentiability, we observe that
\begin{align*}
M(\tau_\ve &(A(\ve))\triangle B)\cr
&\ = \int_{\Nor (F)}\int_0^\infty {\bf 1}(\{ 0< t\le \frac{h_\ve(x,u)\wedge r(x,u)}{\ve}\}\triangle \{0<t\le g(x,u)\})\cr
&\quad\quad\times dt\, \Theta_{d-1}(F,d(x,u))\cr
&\ = \int_{\Nor (F)} \left| \frac{h_\ve(x,u)\wedge r(x,u)}{\ve}-g(x,u)\right| \Theta_{d-1}(F,d(x,u)).
\end{align*}
Since $r(x,u)>0$, for $(x,u)\in\Nor (F)$, the integrand converges to $0$ pointwisely. Also, 
\begin{align*}
\left| \frac{h_\ve(x,u)\wedge r(x,u)}{\ve}-g(x,u)\right| 
& \le \frac{h_\ve(x,u)\wedge r(x,u)}{\ve} + g(x,u)\cr
& \le 2 \max_{0<\ve\le\delta} \frac{h_\ve(x,u)}{\ve} ,
\end{align*}
and the latter function is integrable with respect to $\Theta_{d-1}(F,\cdot)$, by (b).
The Dominated Convergence Theorem thus implies
$$
M(F(\ve)\triangle B)\to 0, \quad \ve \to 0.
$$
This completes the proof of the theorem. \end{proof}

We remark that we could also start with a family $\tilde h_\ve, 0<\ve\le 1$, of functions on $\dF$ and put
$h_\ve (x,u) = \tilde h_\ve (x), (x,u)\in\Nor (F),$ or with a family $\bar h_\ve, 0<\ve\le 1$, of functions on $S^{d-1}$ and put
$h_\ve (x,u) = \bar h_\ve (u), (x,u)\in\Nor (F)$. 

As a particular case, $h_\ve = \ve g$ and the function $g$ could be given by the support function $h_K$ of a convex body $K$ with $0\in K$,
$$
g(x,u) = h_K(u), \quad (x,u)\in \Nor (F).
$$
The subgraph $h_{\ve,{\rm sub}}$, obtained in this case, is different in general from the outer parallel strip $F+\ve K\setminus F$. A differentiability result for outer parallel sets $F+\ve K$, $\ve\to 0$, under different conditions, is discussed in the final Section 6. However, if $K$ is the unit ball $B^d$ and
$$
h_\ve (x,u) =\ve h_{B^d} (u) =\ve ,
$$
then $h_{\ve,{\rm sub}} = F+\ve B^d\setminus F$, as can be easily seen.

A case of particular interest arises, if we choose, in the previous discussion, $h(x,u)=r(x,u)\wedge 1, (x,u)\in \Nor (F)$. If we define, for $\ve >0$, the {\it local parallel set} $F_{\ve, {\rm loc}}$ of $F$ as
$$
F_{\ve, {\rm loc}} = F\cup\{ z=x+tu : (x,u)\in\Nor F, 0<t\le \ve r(x,u)\wedge \ve\} ,
$$
then $F_{\ve, {\rm loc}}$ is the subgraph of $\ve h$. The derivative of $\ve h$ is $r\wedge 1$, hence in the above proof we have
\begin{align*}
M(\tau_\ve &(A(\ve))\triangle B)\cr
&\ = \int_{\Nor (F)} \left| \frac{h_\ve(x,u)\wedge r(x,u)}{\ve}-g(x,u)\right| \Theta_{d-1}(F,d(x,u))\cr
&\ = \int_{\Nor (F)} \left| \frac{\ve (r(x,u)\wedge 1)\wedge r(x,u)}{\ve}-(r(x,u)\wedge 1)\right| \Theta_{d-1}(F,d(x,u))\cr
&\ = 0 ,
\end{align*}
for $\ve\le 1$. Condition (b) is satisfied automatically since $r\wedge 1$ is bounded and integrable with respect to $\Theta_{d-1}(F,\cdot)$ by \eqref{intcond}. Hence, we obtain the following result.

\begin{cor} Let $F$ be a solid set. Then the local parallel set $F_{\ve, {\rm loc}}$, $0<\ve\le 1$, is differentiable at $F$ with derivative
$$
B = \{ (t,x,u) : (x,u)\in\Nor (F), 0\le t\le r(x,u)\wedge 1\}.
$$
\end{cor}

As a consequence, the parallel set $F+\ve B^d$ of a convex body $F$ is differentiable, as we already mentioned above. This is a special case of a result in \cite{Khm07} which shows differentiability of $F+\ve K$, for general convex bodies $F,K$. Our next goal is to extend the latter result to solid sets $F$

For this purpose, we consider the support function $h_K$ of $K$; it can be seen as a continuous function on $S^{d-1}$. We define a function $h_{K,F}$ on $\Nor (F)$ by
$$
h_{K,F}(x,u) = h_K(u), \quad (x,u)\in \Nor (F),
$$
and put
\begin{align*}
(h_{K,F})_{\rm sub} &= \{ (t,x,u)\in \Sigma : 0< t \le h_K(u)\}\cr
&\hspace{2cm}\cup \{ (t,x,u)\in \Sigma : h_K(u)\le t <0\} .
\end{align*}
Notice, that we do not require $0\in K$ here. This is another difference to the discussion of subgraphs above.

In the following theorem, we assume, in addition, that the support measure $\Theta_{d-1}(F,\cdot)$ is finite (this follows, for example, if $\partial F$ has finite $(d-1)$-st Hausdorff measure) and that the set of boundary points of $F$ which are not normal has ${\cal H}^{d-1}$-measure $0$. Here, a point $x\in\dF$ is called {\it normal}, if there is some ball $B\subset F$ with $x\in B$. 
 
\begin{thm} Let $F$ be a solid set with $\Theta_{d-1} (F,\Nor (F))<\infty$ and such that 
$$ {\cal H}^{d-1} (\{ x\in\dF : x {\rm \ not\ normal}\})=0.$$
 Let $K$ be a convex body.
Then $F(\ve) = F+\ve K, 0\le\ve\le 1$, is differentiable at $F$, and we have
$$\frac{d}{d\ve} F(\ve) = (h_{K,F})_{\rm sub} .
$$
\end{thm}

\begin{proof} For $k=1,2,...$, let $\dF_{(k)}$ be the set of all regular boundary points $x$ of $F$ for which there  is a ball of radius $\ge 1/k$ inside $F$ with $x\in B$. Let $u=u(x)$ be the corresponding (outer) normal.  Let $\Sigma_{(k)}\subset \Sigma$ be the  part of the normal cylinder which belongs to points $(x,u(x)), x\in \dF_{(k)}$. We fix $k$ and choose $\ve$ small enough such that $\ve K \subset \frac{1}{k} B^d$. Then, we consider
$$
M([\tau_\ve (F(\ve)\triangle F)\triangle  (h_{K,F})_{\rm sub}]\cap  \Sigma_{(k)}) .
$$
For $x\in \dF_{(k)}$ (with normal $u$), we have $B(1/k) \subset F\subset H(x,u)$, where $B(1/k)$ is the ball of radius $1/k$ touching $F$ at $x$ from inside. $H(x,u)$ is the closure of the complement $\R\setminus C$, where $C$ is the ball of radius $r(x,u)$ touching $F$ in $x$ from outside. If the reach $r(x,u)$ is $\infty$, then $H(x,u)$ is the closed halfspace with outer normal $u$ and containing $x$ in the boundary. We divide $\Sigma_{(k)}$ further into the sets $\Sigma_{(k)}^+$ and $\Sigma_{(k)}^-$ according to the case where $h_{K,F}(x,u) \ge 0$, respectively $h_{K,F}(x,u) < 0$.

Since $\ve K \subset \frac{1}{k} B^d$, we have
$$
[\tau_\ve (F(\ve)\triangle F)\triangle  (h_{K,F})_{\rm sub}]\cap  \Sigma_{(k)}^+ = [\tau_\ve (F(\ve)\setminus F)\triangle  (h_{K,F})_{\rm sub}]\cap  \Sigma_{(k)}^+ 
$$
and
$$
\tau_\ve (F(\ve)\setminus F) \cap  \Sigma_{(k)}^+= \{ (\frac{t}{\ve}, x,u) : 0 < t \le g_{\ve K}(x,u)\} ,
$$
where $g_{\ve K}(x,u)$ is the distance from $x$ to $\partial F(\ve)$ in direction $u$. For $F=H(x,u)$ this distance would be $\ve h_K(u)$, for $F=B(1/k)$ the distance is $\ge \ve h_K(u) + \sqrt{(1/k)^2 - \ve^2 (a(u)^2 - h_K^2(u))}- 1/k$, where $a(u)$ is the maximal length of a point $y\in K$ with $\langle y,u\rangle = h_K(u)$. Hence
$$
 \ve h_K(u) + \sqrt{\frac{1}{k^2} - \ve^2 (a(u)^2 - h_K^2(u))} -\frac{1}{k}\le g_{\ve K}(x,u)\le \ve h_K(u) .$$
We obtain that
\begin{align*}
[&\tau_\ve (F(\ve)\setminus F)\triangle (h_{K,F})_{\rm sub}] \cap  \Sigma_{(k)}\cr
&\subset \left\{ (t, x,u) :  
h_K(u) + \frac{1}{\ve}\left(\sqrt{\frac{1}{k^2} - \ve^2 (a(u)^2 - h_K^2(u))} -\frac{1}{k}\right) \le t \le h_K(u) \right\} .
\end{align*}
Since
$$
\lim_{\ve\to 0}\frac{1}{\ve}\left(\sqrt{\frac{1}{k^2} - \ve^2 (a(u)^2 - h_K^2(u))} -\frac{1}{k}\right) =0 ,
$$
we see that 
$$
M([\tau_\ve (F(\ve)\setminus F)\triangle (h_{K,F})_{\rm sub}] \cap  \Sigma_{(k)}^+)\to 0. 
$$
In a totally analogous way, we obtain that 
$$
[\tau_\ve (F(\ve)\triangle F)\triangle  (h_{K,F})_{\rm sub}]\cap  \Sigma_{(k)}^- = [\tau_\ve (F \setminus F(\ve))\triangle  (h_{K,F})_{\rm sub}]\cap  \Sigma_{(k)}^- 
$$
and 
$$
M([\tau_\ve (F\setminus F(\ve))\triangle (h_{K,F})_{\rm sub}] \cap  \Sigma_{(k)}^-)\to 0.
$$
Hence,
$$
M([\tau_\ve (F\triangle F(\ve))\triangle (h_{K,F})_{\rm sub}] \cap  \Sigma_{(k)})\to 0 ,
$$
for each $k$, and therefore also
$$
M(\tau_\ve (F(\ve)\setminus F)\triangle (h_{K,F})_{\rm sub})\to 0,
$$
as $\ve\to 0$.
\end{proof} 

The conditions on $F$ are fulfilled, in particular, if $F$ is a convex body with interior points, the assumption on the normal boundary points then follows from \cite[Th. 2.5.5]{Sch}. As a corollary, we thus get the following result which was mentioned in \cite{Khm07} (with reference to \cite{Sch}, but without further details).

\begin{cor}
Let $F$ and $K$ be convex bodies und such that $F$ has interior points.
Then $F(\ve) = F+\ve K, 0\le\ve\le 1$, is differentiable at $F$, and we have
$$\frac{d}{d\ve} F(\ve) = (h_{K,F})_{\rm sub} .
$$
\end{cor}

\section{Variations}

The previous considerations show that the concept of differentiability of set-valued functions meets some difficulties, if one takes the step from convex compact bodies  to general compact sets $F$. This is mainly due to the fact that the boundary $\dF$ can have infinite Hausdorff measure ${\cal H}^{d-1}(\dF)=\infty$ and/or to the occurrence of points $(x,u)$ in the normal bundle with arbitrarily small reach $r(x,u)$. As a consequence, the definition of a differentiable family $F(\ve), 0\le\ve\le 1$, is no longer predetermined by the geometrical situation. We have chosen the concept which seems to be the natural extension of the situation for convex bodies. 
In this final section we discuss two variations  which would also lead to a meaningful theory.

First, we can change the essential boundedness condition \eqref{bound}. We call the family $A(\ve),0\le\ve\le 1,$  {\it weakly bounded}, if for each $\delta>0$ there exists a $T>0$ and $\ve_0 =\ve_0(\delta)$ such that
\begin{equation}\label{bound2}
\frac{1}{\ve} \mu_d(A(\ve)\cap (\R\setminus (\dF)_{\ve T})) <\delta,
\end{equation}
for all $\ve<\ve_0$. It is clear that \eqref{bound} implies \eqref{bound2}. Replacing \eqref{bound} by \eqref{bound2} would result in a slightly more general notion of differentiability. For example, in the discussion of subgraphs in Section 6, the condition in (b) that $\max_{0<\ve\le\delta}\frac{h_\ve}{\ve}$ is bounded could be dropped. Thus, integrability would be sufficient to show that $h_{\ve,{\rm sub}}$ is differentiable  at $\dF$. However, for weakly bounded families $A(\ve),0\le\ve\le 1,$ we could no longer assume $A(\ve)\subset (\dF)_{\ve T}$ and also the derivative $B$ would no longer satisfy $B\subset\Sigma_T$. This would require additional estimates in the proofs of the differentiability results which we wanted to avoid.

For a second variation, we remark that, different from the case of convex bodies or sets of positive reach, for a general solid set $F$ it is no longer true that $\mu(A(\ve))\sim \ve M(B(\ve))$ as $\ve\to 0$ (here, $B(\ve) =\tau_\ve (A(\ve))$). For example, it is not true any longer that $\mu_d((\partial F)_{\ve T})$ is of order $\ve M(\Sigma_{T})$ and smallness of one of these values does not imply finiteness of the other. If we want the derivative set $B$ to have finite $M$-measure, then $M(B(\ve))$ has to be controlled separately.

If $A(\ve),0\le\ve\le1,$ is essentially bounded with bound $T>0$, we can assume that $B(\ve)\subset\Sigma_T$. Now let 
$$R_c =\{(x,u)\in \Nor (F): \min( r_+(x,u), r_-(x,u))>c\},$$
for $c>0$ and consider the cylinder $\Sigma_{c,T}=[-T,T]\times R_c$. \\

\noindent
{\bf Definition 3.} Let $F\subset\R$ be a solid set. The set valued function $A(\ve), 0\leq \ve \leq 1,$ is called {\it $r$-differ\-en\-tiable} at $\dF$, with {\it derivative} $B$, if for any fixed $c>0$ and $B(\ve)=\tau_\ve(A(\ve))$
$$M( (B(\ve)\triangle B)\cap \Sigma_{c,T}) \to 0, \; {\rm as} \; \ve\to 0 . $$

\begin{lem}\label{lem:bounded} Suppose $A(\ve), 0\leq \ve \leq 1,$ is differentiable at $\dF$ with derivative $B$. Then it is $r$-differentiable at $\dF$ with the same derivative $B$.
\end{lem}
The reverse statement is not generally true as will be shown by an example below. Therefore, $r$-differentiability is a strictly weaker property and there are more $r$-differentiable  set-valued functions then differentiable ones. In particular, if $F_\ve$ is the parallel set of $F$ then $A(\ve)=F_\ve\setminus F$ is not always differentiable, but it always is $r$-differentiable.

Recall that all measures $|\Theta_{d-j}(F,\cdot))|$ are finite on $R_c$ for any $c>0$.

\begin{lem}\label{lem:bounded2} Suppose $A(\ve), 0\leq \ve \leq 1,$ is $r$-differentiable at $\dF$ with derivative $B$. For $c>0$, let
$$A(\ve,c)=A(\ve)\cap\{z: \min( r_+(p(z),u(z)), r_-(p(z),u(z)))>c\}.$$
 Suppose that the measure $\BP$ satisfies condition \eqref{lambda} and the densities $\bar f_+$ and $\bar f_-$ are integrable with respect to $|\Theta_{d-i}(F,\cdot)|, i=1,\dots,d$, on the set $R_c$. Then
$$
\frac{d}{d\ve}\BP(A({\ve}, c))|_{\ve=0} = \BQ(\frac{d}{d\ve}A({\ve},c) |_{\ve = 0}) = \BQ(B\cap R_{c,T}) .
$$
In particular, if $\BP=\mu_d$ on $F_{\ve T}$, then
$$
\frac{d}{d\ve} \mu_d (A({\ve}, c))|_{\ve=0} = M (\frac{d}{d\ve}A({\ve},c) |_{\ve = 0}) = M(B\cap R_{c,T}) .
$$
\end{lem}

As an example, consider the solid set $F=F_1$ from the example in Section 5. For any $\ve>0$ the parallel set $F_\ve$ and $A(\ve) = F_\ve\setminus F$ contain the rectangle $[-\ve,\ve]\times[0,1]$. The $\Theta_{d-1}(F,\cdot)$ measure of the set 
$$N(\ve) = \{ (x,u)\in\Nor (F) : x\in [-\ve,\ve]\times[0,1]\}$$
is infinite since it is the Hausdorff measure ${\cal H}^{d-1}$ of $\dF\cap [-\ve,\ve]\times[0,1]$, but the integral 
$$\int_{N(\ve)} r_+(x,u)\Theta_{d-1}(F,d(x,u))$$
is finite. The image of $A(\ve)$ under the local magnification map is
$$\tau_\ve(A(\ve))=\{(t,x,u) : 0<t\leq \frac{r_+(x,u)}{\ve}\wedge 1\},$$
since there are no points $z\in \R$ with $d(z)>r(p(z),u(z))$. Therefore 
$$M(\tau_\ve (A(\ve))) = \int_{\Nor(F)} (\frac{r_+(x,u)}{\ve}\wedge 1)  \Theta_{d-1}(F,d(x,u)) <\infty $$
by \eqref{intcond}.
If $A(\ve)$ were differentiable, the derivative should be the set $\Sigma_1= \Nor(F)\times [0,1]$. Since $M(\Sigma_1)=\infty$, the convergence $M(\tau_\ve(A(\ve))\triangle \Sigma_1)\to 0$ cannot be true and, therefore, $A_\ve$ is not differentiable. However, it certainly is $r$-differentiable.

%===========================the end of text======================

\noindent
Authors' addresses: 

\bigskip

\noindent
Est\'ate V. Khmaladze, Victoria University of Wellington, School of Mathematics, Statistics and Operations Research, 
PO Box 600, Wellington, New Zealand, estate.khmaladze@vuw.ac.nz

\bigskip

\noindent
Wolfgang Weil, Karlsruhe Institute of Technology (KIT), 
Department of Mathematics, 
D-76128 Karls\-ruhe, Germany, wolfgang.weil@kit.edu


\begin{thebibliography}{99}

\bibitem{Art} Z.~Artstein, A calculus of set-valued maps and set-valued evolution equations, {\em Set-Valued Anal.} {\bf 3}(1995), 213--261.

%\bibitem{Art2} Z.~Artstein. Invariant measures of set-valued maps, {\em J. Mathemat. Analysis Appl.} {\bf 252} (2000), 696-709.

%\bibitem{Aub} J.-P.~Aubin. Contingent derivatives of set-valued maps and existence of solutions to non-linear inclusions and differential inclusions.In {\em Mathematical Analysis and Applications}, Part A, edited by L.Nachbin, {\em Advances in Mathematics: Supplimentary Studies,} 7A, pp.160-232, Academic Press (1981)

\bibitem{AF} J.-P.~Aubin, H.~Frankowska, {\it Set-valued analysis}, Birkh\"auser, Basel, 1990.

%\bibitem{Ber} F.~Bernardin. Multivalued Stochastic Differential Equation: Convergence of Numerical Scheme, {\em Set-Valued Analysis} {\bf 11} (2003), 393-415.

\bibitem{BZ} J.M.~Borwein, Q.J.~Zhu, \newblock{A survey of sub-differential calculus with applications,} \newblock{\em Nonlinear Analysis} {\bf 38} (1999), 687--773.


\bibitem{BD} B.E.~Brodsky, B.S.~Darkhovsky \newblock{\em Nonparametric Methods in Change-Point Problems}, Kluwer Acad. Publishers, Dordrecht 1993.


\bibitem{CMS} E.~Carlstein,  H.-G.~M\"uller, D. Siegmund (Eds.) {\em Change-point Problems}, IMS Lecture Notes -- Monograph Series {\bf 23}, Inst. Math. Statist., Hayward 1994.


%\bibitem{DM} P.~Deheuvels, D.M.~Mason.\newblock{Functional laws of the iterated logarithm for local empirical processes indexed by sets,}\newblock{\em Ann. Probab.}, {\bf 22} (1994), 1619--1661.

%\bibitem{DM2} P.~Deheuvels, D.M.~Mason.\newblock{Nonstandard local empirical processes indexed by sets,}\newblock{\em  J. Statist. Plann. Inference}, {\bf 45}  (1995), 91--112.

%\bibitem{Ein} J.H.J.~Einmahl.\newblock{Poisson and Gaussian approximation of weighted local empirical processes,}\newblock {\em Stochastic Process. Appl.}, {\bf 70} (1997), 31--58.

\bibitem{EiK}
J.H.J.~Einmahl, E. Khmaladze, \newblock{Central limit theorems for local empirical processes near boundaries of sets,} {\em Bernoulli} {\bf 17} (2011), 545--561. 

%\bibitem{Fed} H.~Federer. \newblock {Curvature measures,}\newblock {\em Trans. Amer. Math. Soc.}, {\bf 93} (1959), 418--491.

%\bibitem{Gau} S.~Gautier.\newblock {Affine and eclipsing multifunctions}, {\em Numer. Funct. Anal. Optim.}, {\bf 11}(1990), 679-699.

\bibitem{HLW}
D.~Hug, G.~Last, W.~Weil,
\newblock {A local Steiner-type formula for general closed sets and applications,}
\newblock {\em Math. Z.} {\bf 246} (2004), 237--272.

%\bibitem{IH} I.A.~Ibragimov, R.Z.~Has'minskii,\newblock {\em Statistical Estimation. Asymptotic Theory},\newblock Springer, New York, 1981.

\bibitem{IM}
B.G. Ivanoff, E. Merzbach, Optimal detection of a change-set in a spatial Poisson process,
{\em Ann. Appl. Probab.} {\bf 20} (2010), 640--659. 

%\bibitem{Khm} E.~Khmaladze.\newblock{Goodness of fit tests for ``chimeric'' alternatives,}\newblock{\em Statist. Neerl.}, {\bf 52} (1998), 90-111.

\bibitem{Khm07}
E.~Khmaladze,
\newblock{Differentiation of sets in measure}, {\em J. Math. Anal. Appl.} {\bf 334} (2007), 1055--1072.

\bibitem{KhMT1}
E.~Khmaladze, R.~Mnatsakanov, N.~Toronjadze,  The change-set problem for Vapnik-\v{C}ervonenkis classes, {\em Mathemat. Methods Statist.} {\bf 15} (2006), 224--231.

\bibitem{KhMT2}
E.~Khmaladze, R.~Mnatsakanov, N.~Toronjadze,  The change-set problem and local covering numbers, {\em Mathemat. Methods Statist.} {\bf 15} (2006), 289--308.
\bibitem{KW}
E.~Khmaladze, W.~Weil,  Local empirical processes near boundaries  of convex bodies.  {\em Ann. Inst. Statist. Math.} {\bf 60} (2008), 813-–842. 


\bibitem{KTs} A. P.~Korostelev, A. B.~ Tsybakov, {\it Minimax Theory of Image Reconstructions}, Lecture Notes in Statist. {\bf 82}, Springer, New York, 1993.



%\bibitem{KK} B. Kyun Kim, Jai Heui Kim. Stochastic Integrals of Set-valued Processes and Fuzzy Processes, {\em J. Mathemat. Analysis Appl.}, {\bf 236} (1999), 480-502. 

\bibitem{LZ} C.~Lemar\'echal and J.~Zowe, The eclipsing concept to approximate a multi-valued mapping, {\em Optimization} {\bf 22} (1991), 3--37.
 
%\bibitem{MuS} M\"uller,  H.G. and Song, K.S. (1996). A set-indexed process in a two-region image. {\it Stochastic Process. Appl.}, {\bf 62}, 87-101.
  
%\bibitem{MS} H.-G.~M\"uller and U.~Stadtm\"uller. \newblock{Discontinuous versus smooth regression}, {\em Ann. Statist.}, {\bf 27} (1999), 299-337.

%\bibitem{Pfl} G.Ch.~Pflug. {\em Optimization of Stochastic Models}, Kluwer Acad. Publ., Dordrecht, 1996.

%\bibitem{RW} R.T.~Rockafellar and R.J.-B.~Werts. {\em Variational Analysis}, Springer, 1998.

\bibitem{Sch}
R.~Schneider, \newblock {\em Convex Bodies: the Brunn-Minkowski Theory},\newblock Encyclopedia of Mathematics and its Applications, {\bf 44}, Cambridge University Press, Cambridge, 1993.

%\bibitem{Shir} A.~Shiryaev, \newblock{\em Probability}, \newblock Springer Verlag, 2004

%\bibitem{Sil} D.~Silin. On set-valued differentiation and integration.{\it Set-Valued Analysis}, {\bf 5} (1997),107-146.

%\bibitem{vVW} A.~van der Vaart, J. A. Wellner, \newblock{\em Weak Convergence of Empirical processes}, \newblock Springer Verlag, 1996

%\bibitem{HW} H.~Weisshaupt. A Measure-Valued Approach to Convex Set-Valued Dynamics, \newblock {\em Set-Valued Analysis}, {\bf 9}, (2001), 337- 373

\end{thebibliography}
\end{document}